\documentclass[reqno,11pt,a4paper, english]{amsart}

\usepackage{amsmath} 
\usepackage{amsthm} 
\usepackage{amssymb} 
\usepackage[dvipsnames]{xcolor}
\usepackage{amscd} 
\usepackage{mathtools}
\usepackage[notcite,notref]{}
\usepackage[ddmmyyyy,hhmmss]{datetime}
\usepackage[colorinlistoftodos,bordercolor=orange,backgroundcolor=orange!20,linecolor=orange,textsize=scriptsize]{todonotes}

\usepackage{subcaption}

\usepackage{array} 
\usepackage{newtxtext}
\usepackage[varvw]{newtxmath}
\usepackage[cal=boondoxo,scr=euler]{mathalfa}
\usepackage[linktocpage]{hyperref} 
\usepackage{cleveref} 
\usepackage{caption} %
\usepackage{graphics,graphicx} 
\usepackage{tikz,tikz-cd} 

\usetikzlibrary{fit,matrix,graphs,graphs.standard,calc,shapes.geometric, arrows}
\usetikzlibrary{decorations.pathreplacing,}

\usetikzlibrary{automata}
\usepackage[shortlabels]{enumitem} 

\DeclareMathAlphabet{\mathsf}{OT1}{\sfdefault}{m}{n}

\newcommand{\nocontentsline}[3]{}
\newcommand{\tocless}[2]{\bgroup\let\addcontentsline=\nocontentsline#1{#2}\egroup}

\usepackage[margin=1.25in]{geometry}
\usepackage{verbatim}

\usepackage{scalerel}

\usepackage{multirow}

\makeatletter
\def\dual#1{\expandafter\dual@aux#1\@nil}
\def\dual@aux#1/#2\@nil{\begin{tabular}{@{}c@{}}#1\\#2\end{tabular}}
\makeatother

\makeatletter
\@namedef{subjclassname@2020}{\textup{2020} Mathematics Subject Classification}
\makeatother

\DeclareMathAlphabet{\amathbb}{U}{bbold}{m}{n}

\hypersetup{
    colorlinks = true,
    linkbordercolor = {white},
    linkcolor = {BrickRed},
    anchorcolor = {black},
    citecolor = {BrickRed},
    filecolor = {cyan},
    menucolor = {BrickRed},
    runcolor = {cyan},
    urlcolor = {black}
}

\tikzstyle{rectan} = [rectangle, rounded corners, 
minimum width=1.5cm, 
minimum height=0.75cm,
text width=3cm,
text centered, 
draw=black,
font = \footnotesize,
]

\tikzstyle{ghost} = [circle, 
minimum width=1pt, 
minimum height=1pt,
text width=1pt,
text centered, 
draw=black,
font = \footnotesize,
]

\newtheoremstyle{teoremas}
{11pt}
{11pt}
{\itshape}
{}
{\bfseries}
{}
{.5em}
{}

\theoremstyle{teoremas}
\newtheorem{theorem}{Theorem}[section]
\newtheorem{corollary}[theorem]{Corollary}
\newtheorem{lemma}[theorem]{Lemma}
\newtheorem{proposition}[theorem]{Proposition}

\newtheoremstyle{definition}
{11pt}
{11pt}
{}
{}
{\bfseries}
{}
{.5em}
{}

\theoremstyle{definition}
\newtheorem{definition}[theorem]{Definition}
\newtheorem{conjecture}[theorem]{Conjecture}

\newtheorem{question}[theorem]{Question}
\newtheorem{example}[theorem]{Example}
\newtheorem{remark}[theorem]{Remark}

\crefname{theorem}{theorem}{theorems}
\Crefname{theorem}{Theorem}{Theorems}
\crefname{lemma}{lemma}{lemmas}
\Crefname{lemma}{Lemma}{Lemmas}
\crefname{proposition}{proposition}{propositions}
\Crefname{proposition}{Proposition}{Propositions}

\DeclareMathOperator{\rk}{rk}

\DeclareMathOperator{\Poin}{Poin}

\newcommand{\M}{\mathsf{M}}
\newcommand{\N}{\mathsf{N}}
\newcommand{\U}{\mathsf{U}}

\newcommand{\Q}{\mathbb{Q}}
\newcommand{\sgn}{\operatorname{sgn}}
\newcommand{\trunc}{\operatorname{trunc}}

\newcommand{\Z}{\mathbb{Z}}
\newcommand{\cC}{\mathcal{C}}

\newcommand{\Hilb}{\operatorname{Hilb}}
\newcommand{\rank}{\operatorname{rk}}

\renewcommand{\H}{\mathrm{H}}
\newcommand{\F}{F} 
\newcommand{\G}{G} %
\newcommand{\CH}{\mathrm{CH}}
\newcommand{\aug}{\operatorname{aug}}
\newcommand{\Int}{\operatorname{Int}}

\newcommand{\uCH}{\underline{\mathrm{CH}}}

\newcommand{\cL}{\mathcal{L}}

\newcommand{\Asc}{\operatorname{Asc}}

\newcommand{\des}{\operatorname{des}}
\newcommand{\Des}{\operatorname{Des}}
\newcommand{\rev}{\operatorname{rev}}
\renewcommand{\aug}{\operatorname{aug}}
\newcommand{\cF}{\mathcal{F}}
\newcommand{\cG}{\mathcal{G}}

\newcommand{\zero}{\widehat{0}}
\newcommand{\one}{\widehat{1}}
\newcommand{\aaa}{\mathbf{a}}
\newcommand{\bbb}{\mathbf{b}}
\newcommand{\wt}{{\sf wt}}

\DeclareMathOperator{\exapsi}{ex\ \aaa\Psi} 
\DeclareMathOperator{\expsitilde}{ex \widetilde{\Psi}} 
\DeclareMathOperator{\exapsib}{ex\ \aaa\Psi\bbb} 
\DeclareMathOperator{\expsib}{ex \Psi\bbb} 

\AtBeginDocument{%
   \def\MR#1{}
}

\title[Dual Chow polynomials]{Dual Chow polynomials of matroids and posets}

\author[G. Caiolo]{Giovanni Caiolo}
\address{(G. Caiolo)
 Scuola Normale Superiore di Pisa, Pisa, Italy
}
\email{giovanni.caiolo@sns.it}

\author[L. Ferroni]{Luis Ferroni}
\address{(L. Ferroni)
   Dipartimento di Matematica, Universit\`a di Pisa, Pisa, Italy
}
\email{luis.ferroni@unipi.it}

\author[E. Hoster]{Elena Hoster}
\address[E.~Hoster]{Fakultät für Mathematik, Ruhr-Universität Bochum, Germany}
\email{elena.hoster@rub.de}

\thanks{}


\allowdisplaybreaks

\begin{document}

\begin{abstract}
 We introduce and study \emph{dual} Chow functions associated to kernels in incidence algebras of weakly ranked posets. Given a kernel $\kappa$, its dual Chow function is defined as the Chow function associated to the sign-twisted reverse kernel. For kernels satisfying a natural skew-symmetry condition, such as the Eulerian kernel of an Eulerian poset or the kernel given by $R$-polynomials on Bruhat intervals, this construction recovers the ordinary Chow function. In contrast, when this skew-symmetry fails, the dual Chow function gives a genuinely different invariant. 
 
 The main example considered in this paper is the dual Chow function associated to the characteristic function. We develop the basic theory of these dual Chow functions, with particular emphasis on posets arising from matroids. We prove chain formulas, unimodality and $\gamma$-positivity results, formulas under standard poset operations, and deletion formulas for matroids.
 Along the way, we also obtain a general deletion formula for the $\aaa\bbb$-index of matroids, which leads to new formulas for extended $\aaa\bbb$-indices and, in turn, specializes to several deletion formulas appearing in the literature.
\end{abstract}

\date{\today~at \currenttime}

\maketitle


\section{Introduction}\label{sec:one}

The Kazhdan--Lusztig--Stanley theory of kernels in incidence algebras provides a common framework for several important polynomial invariants attached to posets. In its original form, going back to Stanley \cite{stanley-local} and Brenti \cite{brenti}, this framework associates to a kernel $\kappa$ two KLS functions, usually called the right and left KLS functions. These functions include, for instance, the Kazhdan--Lusztig polynomials of Bruhat intervals and, more recently, the Kazhdan--Lusztig polynomials of matroids. For a detailed exposition, we refer to Proudfoot's survey \cite{proudfoot-kls}.

Ferroni, Matherne, and Vecchi \cite{ferroni-matherne-vecchi} introduced another invariant associated to a kernel, which they called the \emph{Chow function}. In the special case of the characteristic kernel
\[
    \chi=\mu\cdot\zeta^{\rev},
\]
their construction gives a characteristic Chow polynomial for every weakly ranked poset. When the poset is the lattice of flats of a matroid $\M$, this polynomial agrees with the Hilbert--Poincar\'e series of the Chow ring of $\M$; see \cite{feichtner-yuzvinsky,semismall}. Thus the Chow function of a kernel should be thought of as a common generalization of several cohomological Hilbert series arising from combinatorics and geometry.

The goal of the present paper is to study a dual version of this construction. If $\kappa$ is a $(P,\rho)$-kernel, then its reverse $\kappa^{\rev}$ is again a kernel. However, for the characteristic kernel, this naive reversal introduces signs in a systematic way. For this reason, we define the \emph{dual Chow function} of $\kappa$ to be the Chow function associated to the sign-twisted reverse kernel
\[
    (\kappa^{\rev})^{\sgn},
    \qquad
    \bigl(a^{\sgn}\bigr)_{st}=(-1)^{\rho_{st}}a_{st}.
\]
This convention has the effect of building the natural signs into the definition. In particular, for geometric lattices and, more generally, for Cohen--Macaulay posets, the dual Chow polynomial has nonnegative coefficients.

This definition is harmless in the classical self-dual situations. For two prominent examples, the Eulerian kernel of an Eulerian poset and the kernel of $R$-polynomials on Bruhat order satisfy the \emph{skew-symmetry} relation
\[
    \kappa_{st}(x)=(-1)^{\rho_{st}}x^{\rho_{st}}\kappa_{st}(x^{-1}).
\]
Equivalently, $\kappa^{\rev}=\kappa^{\sgn}$, and hence $(\kappa^{\rev})^{\sgn}=\kappa$. Thus for such kernels the dual Chow function coincides with the ordinary Chow function. The characteristic kernel behaves differently. In that case, the dual Chow function is genuinely new, and this is one of the main subjects of study in the present work.

Our first main results concern arbitrary weakly ranked posets. We prove a chain formula for the dual Chow polynomial, analogous to the Feichtner--Yuzvinsky formula for ordinary Chow polynomials. 

\begin{theorem}[{Theorem~\ref{thm:iterative}}]\label{thm:main1}
    The dual Chow polynomial of a weakly ranked bounded poset $(P,\rho)$ of rank $r$ can be computed as follows:
    \[
        \H^*_P(x)
        =
        (-1)^r
        \sum_{\widehat{0}\leq c_0 < c_1 < \cdots < c_m = \widehat{1}}
        \mu_{\widehat{0}c_0}
        \prod_{i=1}^m
        \mu_{c_{i-1}c_i}
        \frac{x^{\rho_{c_{i-1}c_i}}-x}{x-1}.
    \]
\end{theorem}

The sign $(-1)^r$ appears because the signs contributed by all intervals in a chain add up to the rank of the ambient interval. This formula immediately suggests a positivity statement under certain M\"obius sign conditions. We prove a more general criterion: if the sign-twist of either the right or the left KLS function of a kernel has nonnegative coefficients on every interval, then the sign-twist of its Chow function has nonnegative and unimodal coefficients; see Theorem~\ref{thm:sign-twist-unimodality}. Applying this to the dual characteristic kernel gives the following unimodality result.

\begin{theorem}[{Theorem~\ref{thm:mobius-sign-unimodality}}]\label{thm:main2}
    Let $(P,\rho)$ be a weakly ranked poset such that for every interval $[s,t]$ one has $(-1)^{\rho_{st}}\mu_{st}\geq 0$. Then, for every interval $[s,t]$, the polynomial $\H^*_{st}(x)$ has nonnegative and unimodal coefficients.
\end{theorem}

We also relate dual Chow polynomials to the extended $\aaa\bbb$-index \cite{dorpalenbarry-maglione-stump}. More precisely, after interchanging the roles of $\aaa$ and $\bbb$ in the usual specialization, one obtains the following.

\begin{theorem}[{Theorem~\ref{thm:expsi-to-dualchow}}]\label{thm:main3}
    Let $(P,\rho)$ be a bounded poset of rank $r$. Then
    \[
        \H^*_P(x)=(1-x)^{-r} \expsitilde_P(-x,x,1) \,.
    \]
\end{theorem}

These identities lead to explicit $\gamma$-expansions in terms of the flag $h$-vector.

\begin{theorem}[{Theorem~\ref{thm:gamma exp}}]\label{thm:main4}
    Let $P$ be a bounded weakly ranked poset of rank $r$. Then
    \[
        \H^*_P(x)
        =
        \sum_{\substack{S\subseteq[r-2]\\ S\text{ stable}}}
        \beta_P\bigl( [r-1] \setminus S \bigr)\,
        x^{|S|}(1+x)^{r-1-2|S|} \,.
    \]
    In particular, if $P$ has nonnegative flag $h$-vector, then $\H^*_P(x)$ is $\gamma$-positive. This applies, in particular, to Cohen--Macaulay posets.
\end{theorem}

We then study how dual Chow polynomials behave under standard operations on posets. We discuss joins of posets, augmentation from the top, poset duality, Cartesian products, and truncations. These identities are useful both for computations and for understanding the formal behavior of the invariant.

The matroidal case is especially important. We prove a deletion formula for the dual Chow polynomial of the lattice of flats of a matroid. For a matroid $\M$ on $E$ and an element $i\in E$ that is not a coloop, define $\mathscr{S}_i=\{F\in\cL(\M): i\notin F,\ F\cup\{i\}\in\cL(\M)\}$ and $\underline{\mathscr{S}}_i=\mathscr{S}_i\setminus\{\widehat{0}\}$.

\begin{theorem}[{Theorem~\ref{thm:dual-chow-deletion}}]\label{thm:main5}
    Let $\M$ be a matroid of rank $r$, and let $i\in E$ be an element that is not a coloop and does not have parallel elements. Then
    \[
        \H^*_{\M}(x)
        =
        \H^*_{\M\setminus i}(x)
        +(x+1)\H^*_{\M/i}(x)
        +x\sum_{F\in\underline{\mathscr{S}}_i}
        \H^*_{\M|F}(x)\,
        \H^*_{\M/(F\cup\{i\})}(x).
    \]
\end{theorem}

This deletion formula is obtained by first proving a deletion formula at the level of the classical $\aaa\bbb$-index, and then evaluating using an extension followed by a specialization. These formulas are reminiscent of the semi-small decomposition formula for ordinary Chow polynomials proved in \cite{semismall,ferroni-matherne-stevens-vecchi}, but with different correction terms.

\begin{theorem}[{Theorem~\ref{thm:ab-deletion}}]\label{thm:main6}
    Let $\M$ be a matroid on $E$ and let $i\in E$ be an element that is not a coloop and that does not have parallel elements. The $\aaa\bbb$-index of a matroid $\M$ on $E$ satisfies the following deletion formula:
    \begin{align*}
        \Psi_{\cL(\M)}
        &= \Psi_{\cL(\M\setminus i)} 
            + \bbb \Psi_{\cL(\M/i)} 
            + \sum_{F\in \underline{\mathscr{S}}_i } 
                \Psi_{\cL(\M|F)}\, \aaa\bbb\, \Psi_{\cL(\M/(F\cup\{i\}))} \,.
    \end{align*}
\end{theorem}

We regard the above formula as very significant because it also has applications to other well-studied invariants of matroids; including various versions of the Poincar\'e extended $\aaa\bbb$-index (see Corollary~\ref{thm:extended-ab-deletion} and the discussion on $h$-polynomials of Bergman complexes in Section~\ref{sec:final-remarks}).

One consequence is that $\H^*_{\M}(x)$ is $\gamma$-positive for every matroid $\M$. We also formulate a real-rootedness conjecture for dual Chow polynomials of matroids, in parallel with the Huh--Stevens and Ferroni--Schr\"oter conjecture for ordinary Chow polynomials \cite{stevens-bachelor,ferroni-schroter}. We prove the real-rootedness for dual Chow polynomials of uniform matroids, but we show that there exist EL-shellable posets with non-real-rooted dual Chow polynomial.

Finally, we discuss the problem of interpreting geometrically the dual Chow polynomial. The ordinary Chow polynomial of a matroid is the Hilbert series of the Chow ring $\uCH(\M)$. The dual Chow polynomial cannot be the Hilbert series of a standard graded algebra with one-dimensional degree-zero part. Instead, we ask for \emph{dual Chow module} $\uCH^{\star}(\M)$ built from the Leray model of Bibby--Denham--Feichtner \cite{bibbi-denham-feichtner}, and such that
\[
    \Hilb(\uCH^{\star}(\M),x)=\H^*_{\M}(x).
\]

\subsection*{Acknowledgments}

We are tremendously grateful to Matt Larson and Christian Stump for numerous useful comments provided throughout the preparation of this work, and on earlier versions of this manuscript. LF is a member of the GNSAGA group of the Istituto Nazionale di Alta Matematica (INdAM).

\section{Preliminaries}\label{sec:preliminaries}

Throughout this paper we will largely follow the notation in \cite{ferroni-matherne-vecchi}. However, we give a brief recapitulation about the essentials on Chow functions for partially ordered sets.

\subsection{Incidence algebras}
Henceforth $P$ will always denote a partially ordered set. Recall that~$P$ is said to be \emph{locally finite} if every closed interval
\[
[s,t]=\{w\in P: s\leq w\leq t\}
\]
is finite. The set of all closed intervals of $P$ will be denoted~$\Int(P)$. The \emph{incidence algebra} of~$P$, which we denote by $\mathcal{I}(P)$, is the free $\mathbb{Z}[x]$-module with basis $\Int(P)$. Equivalently, an element $a\in \mathcal{I}(P)$ assigns to each interval $[s,t]$ a polynomial $a_{st}(x)\in \mathbb{Z}[x]$. We often suppress the variable and write simply $a_{st}$ when no confusion can arise. The multiplication in $\mathcal{I}(P)$ is given by convolution:
\[
    (ab)_{st}(x)=\sum_{s\leq w\leq t} a_{sw}(x)\,b_{wt}(x),
    \qquad \text{for every } s\leq t \text{ in } P.
\]

This algebra has the following elementary features.

\begin{enumerate}[(i)]
    \item The convolution product is associative, although it is not commutative in general.
    \item The identity element is the function $\delta\in \mathcal{I}(P)$ given by
    \[
        \delta_{st}=
        \begin{cases}
            1 & \text{if } s=t,\\
            0 & \text{if } s<t.
        \end{cases}
    \]
\end{enumerate}

The following elementary criterion characterizes the invertible elements of an incidence algebra. An element $a\in \mathcal{I}(P)$ has a two-sided inverse $a^{-1}\in \mathcal{I}(P)$ if and only if
    \[
        a_{ss}(x)=\pm 1
    \]
    for every $s\in P$.

A fundamental invertible element in $\mathcal{I}(P)$ is given by the \emph{zeta function} $\upzeta\in \mathcal{I}(P)$ that is defined by
\[
    \upzeta_{st}=1
    \qquad \text{for all } s\leq t.
\]
Its inverse, known as the \emph{M\"obius function}, is denoted by $\mu=\upzeta^{-1}$. The M\"obius function can also be computed recursively by
\[
    \mu_{st}=
    \begin{cases}
        1 & \text{if } s=t,\\
        -\displaystyle\sum_{s\leq w<t}\mu_{sw} & \text{if } s<t.
    \end{cases}
\]

The following lemma, stated in \cite[Lemma~5.3]{ehrenborg-readdy} is a useful tool to compute inverses in incidence algebras.

\begin{lemma}\label{lem:inverses-as-sums-over-chains}
    Let $P$ be a locally finite poset, and let $a\in \mathcal{I}(P)$ be an invertible element such that~$a_{ss}=1$ for all $s\in P$. The element $a^{-1}\in \mathcal{I}(P)$ is given by $(a^{-1})_{ss} = 1$ for all $s\in P$, and
    \[ (a^{-1})_{st} = \sum_{s = s_0 < s_1 < \cdots < s_m = t} (-1)^m\, a_{s_0s_1}\, a_{s_1s_2}\cdots a_{s_{m-1}s_m}\]
    for all $s < t$ in $P$.
\end{lemma}

\subsection{Kernels and KLS functions}

A \emph{weak rank function} on $P$ is a map
\[
    \rho\colon \Int(P)\to \Z_{\geq 0}
\]
such that the following two conditions hold:
\begin{enumerate}[(i)]
    \item If $s<t$, then $\rho_{st}>0$.
    \item If $s\leq w\leq t$, then
    \(
        \rho_{st}=\rho_{sw}+\rho_{wt}.
    \)
\end{enumerate}
In particular, the additivity condition implies $\rho_{ss}=0$ for every $s\in P$. A \emph{weakly ranked poset} is a pair $(P,\rho)$, where $P$ is a poset and $\rho$ is a weak rank function on it. When $P$ has a minimum element $\widehat{0}$, we shall often abbreviate
\(
    \rho(w):=\rho_{\widehat{0},w}.
\)

Given a weak rank function $\rho$ on a locally finite poset $P$, we consider the following subalgebra of $\mathcal{I}(P)$:
\begin{equation}\label{eq:def-of-I-rho-P}
    \mathcal{I}_{\rho}(P)
    =
    \left\{
        a\in \mathcal{I}(P):
        \deg a_{st}(x)\leq \rho_{st}
        \text{ for all } s\leq t \text{ in } P
    \right\}.
\end{equation}
On this subalgebra there is a natural involution $a\mapsto a^{\rev}$, defined interval-wise by
\begin{equation}
    \left(a^{\rev}\right)_{st}(x)
    =
    x^{\rho_{st}}a_{st}(x^{-1}).
\end{equation}
The notation ``rev'' is meant to indicate that this operation reverses the coefficients of the polynomial associated with each interval, relative to the weak rank function. Directly from the definition, one checks that this involution is compatible with convolution:
\(
    (ab)^{\rev}=a^{\rev}b^{\rev}.
\)
Consequently, if $a\in \mathcal{I}_{\rho}(P)$ is invertible, then taking inverses also commutes with this operation:
\(
    (a^{-1})^{\rev}=\left(a^{\rev}\right)^{-1}.
\)

\begin{definition}\label{def:p-kernel}
    Let $(P,\rho)$ be a weakly ranked poset. An element $\kappa\in \mathcal{I}_{\rho}(P)$ is said to be a \emph{$(P,\rho)$-kernel} if $\kappa_{ss}(x) = 1$ for all $s\in P$ and
        \( \kappa^{-1} = \kappa^{\rev}.\)
    The element $\kappa$ is said to be \emph{non-degenerate} if $\deg \kappa_{st} = \rho_{st}$ for every $s\leq t$ in $P$.
\end{definition}

It is not difficult to see that whenever $a\in \mathcal{I}_{\rho}(P)$ is invertible, the elements $\kappa_1 = a^{\rev} a^{-1}$ and $\kappa_2 = a^{-1}a^{\rev}$ are $(P,\rho)$-kernels. A crucial result by Stanley in \cite{stanley-local} shows that in fact all kernels arise are of this form. Furthermore, he showed the following (see \cite[Theorem~2.2]{proudfoot-kls} for a short proof).

\begin{theorem}\label{thm:kls-functions}
    Let $\kappa\in \mathcal{I}_{\rho}(P)$ be a $(P,\rho)$-kernel. Then there is a unique element $f\in \mathcal{I}_{\rho}(P)$ with the following properties:
    \begin{enumerate}[\normalfont(i)]
        \item \label{it:f-i} $f_{ss}(x)=1$ for every $s\in P$.
        \item \label{it:f-ii} $\deg f_{st}(x)<\frac{1}{2}\rho_{st}$ whenever $s<t$.
        \item \label{it:f-iii} $\kappa=f^{\rev}\cdot f^{-1}$.
    \end{enumerate}
    Likewise, there is a unique element $g\in \mathcal{I}_{\rho}(P)$ satisfying:
    \begin{enumerate}[\normalfont(i')]
        \item \label{it:g-i} $g_{ss}(x)=1$ for every $s\in P$.
        \item \label{it:g-ii} $\deg g_{st}(x)<\frac{1}{2}\rho_{st}$ whenever $s<t$.
        \item \label{it:g-iii} $\kappa=g^{-1}g^{\rev}$.
    \end{enumerate}
\end{theorem}

Adopting the terminology of \cite{brenti} and \cite[Section~2]{proudfoot-kls}, we refer to $f$ as the \emph{right Kazhdan--Lusztig--Stanley (KLS) function} associated with $\kappa$, and to $g$ as the corresponding \emph{left Kazhdan--Lusztig--Stanley (KLS) function}. When $P$ is bounded, the polynomial
\[
    f_P(x):=f_{\widehat{0}\,\widehat{1}}(x)
\]
is called the \emph{right Kazhdan--Lusztig--Stanley polynomial} of $P$, while
\[
    g_P(x):=g_{\widehat{0}\,\widehat{1}}(x)
\]
is called its \emph{left Kazhdan--Lusztig--Stanley polynomial}.

\subsection{Chow functions} As was shown in \cite[Section~3]{ferroni-matherne-vecchi}, each $(P,\rho)$-kernel $\kappa$ induces an element $\H\in \mathcal{I}_{\rho}(P)$ satisfying remarkable properties. The definition is as follows. First, note that the condition on $\kappa$ being a kernel imposes that $\kappa_{st}(x)$ is divisible by $x-1$ whenever $s < t$. Then, one defines the object $\overline{\kappa}\in \mathcal{I}_{\rho}(P)$ by
   \[ \overline{\kappa}_{st}(x) 
    = \begin{cases}
        \frac{1}{x-1}\, \kappa_{st}(x) & \text{ if $s < t$}\\
        -1 & \text{ if $s = t$}.
    \end{cases}\]

\begin{definition}\label{def:chow-function}
    Let $\kappa$ be a $(P,\rho)$-kernel. The \emph{Chow function} associated to $\kappa$, or the \emph{$\kappa$-Chow function}, is defined as the element $\H\in \mathcal{I}_{\rho}(P)$ defined by
    \[ \H = - \left(\overline{\kappa}\right)^{-1}.\]
    If the poset $P$ is bounded, the polynomial $\H_P(x)=\H_{\widehat{0}\,\widehat{1}}(x)$ is customarily called the \emph{$\kappa$-Chow polynomial} of the poset.
\end{definition}

Chow functions satisfy several important properties that we now recapitulate.

\begin{proposition}[{\cite[Proposition~3.4]{ferroni-matherne-vecchi}}]\label{prop:degree-and-symmetry}
    Let $\kappa$ be a $(P,\rho)$-kernel, and let $\H\in \mathcal{I}_{\rho}(P)$ be the corresponding Chow function. The following properties hold.
    \begin{enumerate}[\normalfont(i)]
        \item\label{it:degree-of-H} For every $s < t$, we have that
            \[ [x^{\rho_{st}-1}] \H_{st}(x) = [x^{\rho_{st}}] \kappa_{st}(x).\]
            In particular, if $\kappa$ is non-degenerate, then $\deg \H_{st}(x) = \rho_{st} - 1$ for every $s < t$.
        \item \label{it:chow-symmetry}The Chow function is symmetric, i.e.,
            \[ \H_{st}(x) = x^{\rho_{st} - 1}\, \H_{st}(x^{-1}) \qquad \text{ for every $s < t$}.\]
    \end{enumerate}
\end{proposition}

\begin{theorem}[{\cite[Theorem 3.12]{ferroni-matherne-vecchi}}]
\label{thm:unimodal-nonneg}
    Let $\kappa$ be a $(P, \rho)$-kernel. 
    If either the right KLS function $f$ or the left KLS function $g$ are non-negative, 
    then the $\kappa$-Chow function is non-negative and unimodal.
\end{theorem}

\subsection{Augmented Chow functions and \texorpdfstring{$Z$}{Z}-functions}

Another important class of functions associated to a kernel are the following:

\begin{definition}
    Let $\kappa$ be a $(P,\rho)$-kernel. Consider the following two elements of $\mathcal{I}_{\rho}(P)$:
        \begin{align*}
            F &= \H \cdot f^{\rev},\\
            G &= g^{\rev}\cdot \H.
        \end{align*}
    We call $F$ (resp. $G$) \emph{the right (resp. left) augmented Chow function associated to $\kappa$}. Furthermore, we define the $Z$-function associated to $\kappa$ as the element $Z\in \mathcal{I}_{\rho}(P)$ given by
        \[ Z = g^{\rev}f = gf^{\rev}\]
\end{definition}

The three elements $F$, $G$, and $Z$ are invariant under the rev involution, i.e., $F^{\rev} = F$, $\G^{\rev} = G$, and $Z^{\rev} = Z$. For a detailed study of $Z$-functions we refer to \cite{proudfoot-kls}.

\subsection{Characteristic functions}

Every weakly ranked poset $(P,\rho)$ comes endowed with a canonical kernel. 

\begin{definition}
The \emph{characteristic function} associated with $(P,\rho)$ is the element $\chi\in \mathcal{I}_{\rho}(P)$ defined by
\begin{equation}\label{eq:char-function}
    \chi
    =
    \mu\cdot \upzeta^{\rev}
    =
    \upzeta^{-1}\cdot \upzeta^{\rev}.
\end{equation}
\end{definition}

For each interval $[s,t]$ of $P$, one may express the characteristic function in the following handy form.
\[
    \chi_{st}(x)
    =
    \sum_{s\leq w\leq t}\mu_{sw}\,x^{\rho_{wt}}.
\]
Whenever $P$ is bounded, we shall often refer to the polynomial $\chi_P(x):=\chi_{\widehat{0}\,\widehat{1}}(x)$
as the \emph{characteristic polynomial of $P$}.

It is immediate from the definitions that the \emph{left} KLS function associated to the kernel $\kappa = \chi$ is simply $g = \zeta$. The right KLS function $f$ is a very subtle object, even in the case in which the poset $P$ satisfies highly restrictive properties, such as being a geometric lattice (in fact, this polynomial plays a prominent role in the singular Hodge theory of matroids \cite{elias-proudfoot-wakefield,braden-huh-matherne-proudfoot-wang}).
The Chow function $\H$ associated to $\kappa = \chi$ is called the \emph{characteristic Chow function} or \emph{$\chi$-Chow function} of the weakly ranked poset $(P,\rho)$. 

\subsection{The extended \texorpdfstring{$\aaa\bbb$}{ab}-index}\label{subsec:prelim-ab-index}
Let
\(
    R=\mathbb{Z}[y]\langle \aaa,\bbb\rangle
\)
be the ring of polynomials in the two noncommuting variables $\aaa$ and $\bbb$, with coefficients in $\mathbb{Z}[y]$. In this subsection we regard incidence functions with values in $R$; that is, elements of the $R$-valued incidence algebra~$\mathcal{I}(P;R)$.

We shall assume throughout this subsection that every interval of $P$ is graded with respect to the weak rank function $\rho$.
In particular, every interval $[s,t]$ in $P$ is a graded subposet with respect to the weak rank function $\rho$, and $\rho_{st}$ is the length of a maximal chain in $[s,t]$.

For a chain $\cC=\{ c_1 < \cdots < c_m \} $ and $s\leq t$ in $P$,
define the \emph{weight} $\wt_{st}^\cC(\aaa,\bbb)\in \mathcal{I}(P;R)$ to be $0$ if $\cC$ is not fully contained in the interval $[s,t]$, and otherwise as
\[
    \wt_{st}^\cC(\aaa,\bbb)
    =
    w_1\cdots w_{r-1} \,,\qquad \text{where $r:=\rho_{st}$,}
\]
by declaring
\[
    w_i=
    \begin{cases}
        \bbb & \text{if } i\in\{\rho_{s c_1},\ldots,\rho_{s c_m} \},\\
        \aaa-\bbb & \text{otherwise}.
    \end{cases}
\]
Notice that neither the rank of the lower bound $s$ nor the rank of the upper bound $t$ are recorded in this word.
We shall write~$\wt^\cC(\aaa,\bbb)$ if the interval $[s,t]$ is clear from context.

\begin{definition}
    The $\aaa\bbb$-index of a poset $P$ is the element
    \(
        \Psi\in \mathcal{I}(P;R)
    \)
    defined interval-wise by
    \begin{equation}
    \label{eq:def-ab-index}
        \Psi_{st}(\aaa,\bbb) 
        = \sum_{\cC} \wt_{st}^{\cC}(\aaa,\bbb)
    \end{equation}
    where the sum ranges over all chains $\cC$ in the open interval~$(s,t)=\{w\in P \mid s < w < t\}$.
\end{definition}
Notice again that the ranks of $s$ and $t$ are not recorded in the definition of the $\aaa\bbb$-index.
Recording these ranks corresponds to appending $\aaa$ or $\bbb$ from left or right; see \Cref{rem:ab-index}.
With the usual notation for the flag $h$-vector (see, e.g., \cite[Chapter~3]{stanley-ec1}), we also have
    \[
        \Psi_P(\aaa,\bbb)
        =
        \sum_{S\subseteq[r-1]}\beta_P(S)\operatorname{m}_S(\aaa,\bbb)\,,
    \]
where $\operatorname{m}_S = m_1\ldots m_{r-1}$ is defined by declaring $m_i=\bbb$ if $i\in S$ and $m_i=\aaa$ otherwise.

We now define the extended $\aaa\bbb$-index as (one of the variants of) the $\aaa\bbb$-index extended by a $\mathbb{Z}$-linear transformation $\omega$. 

\begin{definition}
    Let $\omega:\mathbb{Z}\langle \aaa,\bbb\rangle \to \mathbb{Z}[y]\langle \aaa,\bbb\rangle$ be the $\mathbb{Z}$-linear  transformation that replaces all occurrences of $\aaa\bbb$ by $(1+y)\cdot (\aaa\bbb + y\bbb\aaa)$ and all remaining occurrences of $\aaa$ by $\aaa+y\bbb$ and of $\bbb$ by $\bbb+y\aaa$. 
    The \emph{extended $\aaa\bbb$-index} is the element
    \(
        \exapsi\in \mathcal{I}(P;R)
    \)
    defined for every interval $[s,t]$ with $s<t$ by
    \[
        \exapsi_{st}(y,\aaa,\bbb)
        =
        \omega ( \aaa \Psi_{st}(\aaa,\bbb)) \,.
    \] 
    If $s=t$ we set $\exapsi _{ss}(y,\aaa,\bbb)=1$.\\
    If $P$ is bounded, the \emph{extended $\aaa\bbb$-index of $P$} is defined as
    \(
        \exapsi_P(y,\aaa,\bbb)
     :=
        \exapsi_{\widehat{0}\,\widehat{1}}(y,\aaa,\bbb).
    \)
    We further define
    \(
        \expsitilde \in \mathcal{I}(P;R)
    \)
    interval-wise by
    \[
        \expsitilde_{st}(y,\aaa,\bbb)
        :=
        (1+y) \cdot \omega( \Psi_{st}(\aaa,\bbb)) \,,
    \]
    where again we set $\expsitilde _{ss}(y,\aaa,\bbb)=1$ and write
    \(
        \expsitilde_P(y,\aaa,\bbb) := \expsitilde_{\zero, \one}(y,\aaa,\bbb)
    \)
    if $P$ is bounded.
\end{definition}

Let us consider the $\mathbb{Z}[y]$-linear map $\iota$ on $\mathbb{Z}[y]\langle \aaa,\bbb\rangle$ that deletes the leftmost letter in every $\aaa\bbb$-monomial, and with the convention $\iota(1)=1$.
This transformation interacts with the $\omega$-transformation as
\begin{align}
\label{eq:iota-omega}
    \iota \circ \omega = (1+y) \cdot \omega \circ \iota \,.
\end{align}
In particular, 
\(
    \expsitilde_{st}(y,\aaa,\bbb)
    =
    \iota\bigl(\exapsi_{st}(y,\aaa,\bbb)\bigr) \,.
\)

The notation that we use in this paper differs slightly from the literature on the extended $\aaa\bbb$-index such as \cite{hoster-stump-vecchi,dorpalenbarry-maglione-stump,stump}.
As a general rule, we denote
\begin{align}
\label{eq:ex-and-tilde}
    \operatorname{ex} f = \omega (f)
    \quad \text{and} \quad
    \operatorname{ex} \widetilde{f} = (1+y) \cdot \operatorname{ex} f\,,
\end{align}
for $f\in \mathbb{Z}\langle \aaa,\bbb\rangle$ a minor modification of the $\aaa\bbb$-index.

As shown in \cite{hoster-stump-vecchi}, the $\omega$-extended definition of the extended $\aaa\bbb$-index is equivalent to the original \emph{Poincaré-extended} definition due to Dorpalen-Barry--Maglione--Stump in \cite{dorpalenbarry-maglione-stump}.
For every interval $[s,t]$ in a graded poset $P$, define its Poincar\'e polynomial by
\[
    \Poin_{st}(y)
    :=
    \chi^{\rev}_{st}(-y)
    =
    (-y)^{\rho_{st}}\chi_{st}(-y^{-1}).
\]
Consider a chain
\(
    \cC=\{c_1 < \cdots < c_k\}
\)
in the interval $[s,t]$.
Define the \emph{chain Poincar\'e polynomial} associated to $\cC$ by
\[
    \Poin_{st}^\cC(y)
    :=
    \prod_{i=1}^{k}\Poin_{c_i c_{i+1}}(y),
\]
with the convention that $c_{k+1}=t$ and the empty product is equal to $1$.
Note that~$\Poin_{st}^{\{s,t\}}(y) = \Poin_{st}(y)$.
As before, we write $\Poin_P^{\cC}(y) = \Poin_{\zero \one}^\cC(y)$ if $P$ is bounded.

\begin{theorem}
\label{thm:expsi-via-poincare}
    The extended $\aaa\bbb$-index is given by
    \[
        \exapsi_{st}(y,\aaa,\bbb)
        =
        \sum_{\cC \text{ in } [s,t)} \Poin_{st}^{\cC}(y) \ w_0^\cC \wt^{\cC}_{st}(\aaa,\bbb) \,.
    \]
    where the sum ranges over all chains $\cC$ in the half-open interval $[s,t)=\{w\in P \mid s\leq w < t\}$, and where $w_0^\cC = \bbb$ if $s\in \cC$ and $w_0^\cC = \aaa - \bbb$ otherwise.
    Applying $\iota$ yields
    \[
        \expsitilde_{st}(y,\aaa,\bbb)
        =
        \sum_{\substack{\cC \text{ in } [s,t) \\ s\in \cC}} \Poin_{st}^{\cC}(y) \ \wt_{st}^{\cC}(\aaa,\bbb) \,,
    \]
    where the sum ranges over all chains $\cC$ in the open interval $[s,t)$ with $s\in \cC$.
\end{theorem}

The following result due to Stump \cite{stump} shows that extended $\aaa\bbb$-index specializes to the $\chi$-Chow function and the left augmented~$\chi$-Chow function in a very nice way.

\begin{theorem}\label{thm:stump}
    Let $(P,\rho)$ be a weakly ranked poset. For every $s\leq t$ the following holds:
        \begin{align*} \expsitilde_{st}(-x,1,x) &= (1-x)^{\rho_{st}}\,\H_{st}(x)\\
        \exapsi_{st}(-x,1,x) &= (1-x)^{\rho_{st}}\,G_{st}(x)\\        \end{align*}
\end{theorem}

\begin{remark}\label{rem:ab-index}
    Using the $\aaa\bbb$-index as we define it here we can straightforward create variants with restrictions to the bounds $s$ and $t$.
    We will later focus on the following three variants of the $\aaa\bbb$-index:
    \begin{enumerate}[leftmargin=2.5cm]
        \item[$\Psi_{st}(\aaa,\bbb)$] is
        the $\aaa\bbb$-index as defined in \eqref{eq:def-ab-index};
        \item[$\aaa\Psi_{st}(\aaa,\bbb)$]
        records the rank of $s$ in the weight of a chain $\cC$ in the half-open interval $[s,t)$ but $\cC$ may or may not contain the element $s$; and
        \item[$\Psi_{st}(\aaa,\bbb)\bbb$]
        records the rank of the endpoint $t$ of chains $\cC \cup \{t\}$ in the half-open interval~$(s,t]$ ending at $t$\,.
    \end{enumerate}
\end{remark}

\section{Dual Chow functions for posets}

\subsection{Dual Chow functions for arbitrary kernels}

For an element $a\in\mathcal{I}_{\rho}(P)$, define its \emph{sign twist} $a^{\sgn}$ by
\[
    \left(a^{\sgn}\right)_{st}(x)
    =
    (-1)^{\rho_{st}}a_{st}(x).
\]
Since the weak rank function is additive on intervals, the sign twist is an algebra automorphism of $\mathcal{I}_{\rho}(P)$; in particular,
\[
    (ab)^{\sgn}=a^{\sgn}b^{\sgn},
    \qquad
    (a^{-1})^{\sgn}=(a^{\sgn})^{-1}.
\]
It is also immediate that the sign twist commutes with the rev involution:
\[
    (a^{\sgn})^{\rev}=(a^{\rev})^{\sgn}.
\]

If an element $\kappa\in \mathcal{I}_{\rho}(P)$ is a $(P,\rho)$-kernel, then so is the element $(\kappa^{\rev})^{\sgn}$. In particular, we can consider the Chow function associated to $(\kappa^{\rev})^{\sgn}$.

\begin{definition}
    Let $\kappa$ be a $(P,\rho)$-kernel on a poset of rank $r$. The \emph{dual Chow function} associated to $\kappa$, denoted by $\H^*$, is the Chow function associated with the kernel $(\kappa^{\rev})^{\sgn}$. If $P$ is bounded, we call 
    \[\H^*_{P}(x) = \H^*_{\widehat{0}\widehat{1}}(x)\]
    the \emph{dual $\kappa$-Chow polynomial} of $P$.
\end{definition}

Note that in many important cases in mathematics, one deals with a kernel $\kappa$ satisfying
\begin{equation}\label{eq:sym-kernel}
    \kappa_{st}(x)
    =
    (-1)^{\rho_{st}}\,x^{\rho_{st}}\kappa_{st}(x^{-1})
    =
    (-1)^{\rho_{st}}\kappa^{\rev}_{st}(x),
    \qquad \text{for each $s\leq t$.}
\end{equation}
Two fundamentally important examples in which this happens are the following:
\begin{itemize}
    \item When $P$ is an Eulerian poset and $\kappa = \varepsilon$ is the Eulerian kernel, given by $\varepsilon_{st} = (x-1)^{\rho_{st}}$ for each $s\leq t$.
    \item When $P$ is the Bruhat poset on a Coxeter group and the kernel $\kappa$ is given by the $R$-polynomials.  
\end{itemize}
The condition in \eqref{eq:sym-kernel} is equivalent to
\(
    \kappa^{\rev}=\kappa^{\sgn}.
\)
Hence
\[
    (\kappa^{\rev})^{\sgn}=\kappa.
\]
Thus, for skew-symmetric kernels such as the Eulerian kernel or the $R$-polynomial kernel, the Chow function and the dual Chow function agree. Note, however, that for the characteristic kernel $\kappa=\chi$, the condition in \eqref{eq:sym-kernel} does not hold. In particular, for characteristic Chow functions, the dual Chow function and the ordinary Chow function are substantially different.

\begin{remark}
    Consider a weakly ranked poset $(P,\rho)$ and its dual $(P^*,\rho^*)$, where $P^*$ is obtained by reversing the comparability relations in $P$ and $\rho^*_{ts} = \rho_{st}$ for each $s\leq t$ in $P$. The kernel $\kappa$ induces a kernel $\kappa^*$ in $(P^*,\rho^*)$ by defining $\kappa^*_{ts} = \kappa_{st}$ for each $s\leq t$ in $P$. The reader should not confuse the Chow $\kappa^*$-Chow function of $(P^*,\rho^*)$ with the  dual Chow function $\H^*$ associated to~$\kappa$ in $(P,\rho)$. These objects are not immediately related. 
\end{remark}

\subsection{Dual Chow functions and KLS functions}

We will henceforth assume that $\kappa$ is a $(P,\rho)$-kernel, with Chow function $\H$, right and left KLS functions $f$ and $g$, and $Z$-function $Z$. Correspondingly, $\H^*$ is the dual Chow function associated to $\kappa$, i.e., the Chow function associated to $(\kappa^{\rev})^{\sgn}$; also, $f^*$ and $g^*$ denote the right and left KLS functions, and $Z^*$ denotes the $Z$-function of $(\kappa^{\rev})^{\sgn}$.

\begin{proposition}\label{prop:kls-and-zeta-of-rev-kernel}
    Let $\kappa$ be a $(P,\rho)$-kernel. The following holds:
    \[ f^* = (g^{-1})^{\sgn}\, \qquad \text{ and } \qquad g^* = (f^{-1})^{\sgn}.\]
    In particular $Z^*=(Z^{-1})^{\sgn}$.
\end{proposition}

\begin{proof}
    Let us prove only that $f^*=(g^{-1})^{\sgn}$ as the other equality is very similar. Note that $g$ satisfies $\kappa=g^{-1}g^{\rev}$. In particular,
    \[
        (\kappa^{\rev})^{\sgn}
        =
        \left((g^{-1})^{\rev}g\right)^{\sgn}
        =
        \left((g^{-1})^{\sgn}\right)^{\rev}
        \left((g^{-1})^{\sgn}\right)^{-1}.
    \]
    The degree bounds for $(g^{-1})^{\sgn}$ can be proved with a straightforward induction: from $g^{-1}g=\delta$ one writes
    \[
        (g^{-1})_{st}(x)=-\sum_{s\leq w<t}(g^{-1})_{sw}(x)g_{wt}(x),
    \]
    and uses the degree bounds for $g$. Hence $(g^{-1})^{\sgn}$ is the right KLS function of $(\kappa^{\rev})^{\sgn}$.

    The conclusion about the $Z$-function is direct from the formulas above:
    \[
        Z^*=(g^*)^{\rev}f^*=((f^{\rev})^{-1}g^{-1})^{\sgn}=(Z^{-1})^{\sgn}.\qedhere
    \]
\end{proof}

\subsection{Dual augmented Chow functions}

Let us denote by $F$ and $G$ the right and left augmented Chow functions associated to $\kappa$, and let us denote by $\F^*$ and $\G^*$ the right and left augmented Chow functions associated to $(\kappa^{\rev})^{\sgn}$. We will call $\F^*$ the \emph{dual right augmented Chow function} and $\G^*$ the \emph{dual left augmented Chow function} associated to $\kappa$.

\begin{proposition}\label{prop:augmented-chow-of-rev-kernel}
    The dual augmented Chow functions satisfy the identities
    \[
        \F^*\G^{\sgn}=\H^*\H^{\sgn},
        \qquad
        \F^{\sgn}\G^*=\H^{\sgn}\H^*.
    \]
    Equivalently,
    \[
        \F^*=\H^*\H^{\sgn}(\G^{\sgn})^{-1},
        \qquad
        \G^*=(F^{\sgn})^{-1}\H^{\sgn}\H^*.
    \]
\end{proposition}

\begin{proof}
    By Proposition~\ref{prop:kls-and-zeta-of-rev-kernel}, we have $f^*=(g^{-1})^{\sgn}$ and $g^*=(f^{-1})^{\sgn}$. Thus, by definition of the right dual augmented Chow function,
    \[
        \F^*=\H^*(f^*)^{\rev}=\H^*((g^{\rev})^{-1})^{\sgn}.
    \]
    On the other hand, since $\G=g^{\rev}\H$, we have $(g^{\rev})^{-1}=\H \G^{-1}$. Hence
    \[
        \F^*=\H^*\H^{\sgn}(\G^{\sgn})^{-1},
    \]
    which is equivalent to $\F^*\G^{\sgn}=\H^*\H^{\sgn}$.

    Similarly,
    \[
        \G^*=(g^*)^{\rev}\H^*=((f^{\rev})^{-1})^{\sgn}\H^*.
    \]
    Since $F=\H f^{\rev}$, we have $(f^{\rev})^{-1}=F^{-1}\H$. Therefore $\G^*=(F^{\sgn})^{-1}\H^{\sgn}\H^*$, which is equivalent to $F^{\sgn}\G^*=\H^{\sgn}\H^*$.
\end{proof}

\subsection{Sign twists and unimodality for (dual) Chow functions}

We shall use the following consequence of the canonical decompositions for Chow functions proved in \cite{ferroni-matherne-vecchi}: if a kernel has either its right or left KLS function with nonnegative coefficients on every interval, then its Chow function has nonnegative and unimodal coefficients on every interval.

\begin{theorem}\label{thm:sign-twist-unimodality}
    Let $\kappa$ be a $(P,\rho)$-kernel. If either $f^{\sgn}$ or $g^{\sgn}$ has nonnegative coefficients on every interval, then the sign twisted Chow function $\H^{\sgn}$ has nonnegative and unimodal coefficients on every interval.
\end{theorem}

\begin{proof}
    Consider the sign-twisted kernel $\kappa^{\sgn}$. Since the sign twist is an algebra automorphism and commutes with rev, $\kappa^{\sgn}$ is again a $(P,\rho)$-kernel. The right and left KLS functions of $\kappa^{\sgn}$ are $f^{\sgn}$ and $g^{\sgn}$, respectively, by uniqueness of KLS functions. Also, since
    \[
        \overline{\kappa^{\sgn}}=(\overline{\kappa})^{\sgn},
    \]
    the Chow function of $\kappa^{\sgn}$ is $\H^{\sgn}$. If either $f^{\sgn}$ or $g^{\sgn}$ is nonnegative, the nonnegativity and unimodality criterion for Chow functions, \Cref{thm:unimodal-nonneg}, applied to the kernel $\kappa^{\sgn}$ proves the result.
\end{proof}

\begin{corollary}\label{cor:dual-sign-twist-unimodality}
    Let $\kappa$ be a $(P,\rho)$-kernel, and let $\H^*$ be its dual Chow function. Let $f^*$ and $g^*$ be the right and left KLS functions of the kernel $(\kappa^{\rev})^{\sgn}$. If either $f^*$ or $g^*$ has nonnegative coefficients on every interval, then $\H^*$ has nonnegative and unimodal coefficients on every interval.
\end{corollary}

\begin{proof}
    This is the nonnegativity and unimodality criterion for Chow functions applied to the kernel $(\kappa^{\rev})^{\sgn}$.
\end{proof}

\section{Dual Chow polynomials for the characteristic function}

In this section we shall be concerned only with the case in which $\kappa = \chi$ is the characteristic function of $(P,\rho)$, and all the Chow functions, dual Chow functions, and KLS functions alluded to in this section refer only to this case. For the sake of a cleaner presentation, we will compute dual Chow polynomials of posets rather than dual Chow functions; there is no difficulty in translating our statements to the full dual Chow function. One important caveat is that, unlike~$\chi$, the kernel~$(\chi^{\rev})^{\sgn}$ can be degenerate. 

Note that the dual right KLS function associated to $(\chi^{\rev})^{\sgn}$ is the sign-twisted M\"obius function $\mu^{\sgn}$. The \emph{left} KLS function associated to $(\chi^{\rev})^{\sgn}$ is $g^* = (f^{-1})^{\sgn}$, where $f$ is the right KLS function associated to $\chi$. The objects $f$ and $Z$ are notoriously difficult, and have been studied in the case in which the poset in question is a matroid (see \cite{gao-xie,gao-xie-yang,braden-huh-matherne-proudfoot-wang,gao-ruan-xie,braden-ferroni-matherne-nepal}). 

\subsection{Non-recursive formulas, nonnegativity and unimodality}
The following formula provides a ``dual'' of the Feichtner--Yuzvinsky formula for Chow polynomials.

\begin{theorem}\label{thm:iterative}
    The dual Chow polynomial of a weakly ranked bounded poset $(P,\rho)$ of rank $r$ can be computed as follows:
    \[
        \H^*_P(x)
        =
        (-1)^r
        \sum_{\widehat{0}\leq c_0 < c_1 < \cdots < c_m = \widehat{1}}
        \mu_{\widehat{0}c_0}
        \prod_{i=1}^m
        \mu_{c_{i-1}c_i}
        \frac{x^{\rho_{c_{i-1}c_i}}-x}{x-1}.
    \]
    More generally, for every interval $[s,t]$ one has
    \[
        \H^*_{st}(x)
        =
        (-1)^{\rho_{st}}
        \sum_{s\leq c_0 < c_1 < \cdots < c_m = t}
        \mu_{s c_0}
        \prod_{i=1}^m
        \mu_{c_{i-1}c_i}
        \frac{x^{\rho_{c_{i-1}c_i}}-x}{x-1}.
    \]
\end{theorem}

\begin{proof}
    We use the formula appearing in \cite[eq.~(17)]{ferroni-matherne-vecchi}, together with the fact that the right KLS function associated to $(\chi^{\rev})^{\sgn}$ is $f^*=\mu^{\sgn}$. This gives
    \[
    \H^*_{st}(x)
    =
    \sum_{s \leq c_0 < c_1 < \dots < c_m=t}
    (-1)^{\rho_{s c_0}}\mu_{s c_0}
    \prod _{i=1}^m
    (-1)^{\rho_{c_{i-1}c_i}}\mu_{c_{i-1}c_i}
    \frac{x^{\rho_{c_{i-1}c_i}}-x}{x-1}.
    \]
    Since the weak rank function is additive, the exponents of $-1$ appearing in each summand add up to
    \[
        \rho_{s c_0}+\sum_{i=1}^m\rho_{c_{i-1}c_i}
        =
        \rho_{st}.
    \]
    Factoring out this common sign gives the interval formula. Taking $s=\widehat{0}$ and $t=\widehat{1}$ gives the stated formula for $\H^*_P(x)$.
\end{proof}

Unlike its non-dual counterpart, the formula in the preceding proposition may contain many signs. It is natural to inquire whether the result is in the end always a polynomial with all nonnegative coefficients. The answer is negative, as shown by the following example.

\begin{example}
    Consider the rank $3$ graded poset $P$ depicted in Figure~\ref{fig:counterex-positivity}. The dual Chow polynomial $\H^*_P(x)$ can be computed by hand, yielding 
        \[ \H^*_P(x) = 1 - 2x +x^2.\]
    
    \begin{figure}[ht]
    \centering
	\begin{tikzpicture}  
	[scale=0.7,auto=center,every node/.style={circle,scale=0.8, fill=black, inner sep=2.7pt}] 
	\tikzstyle{edges} = [thick];
	
	\node[] (s)  at (0,-1) {};
	
	\node[] (a1) at (-3,1) {};
	\node[] (a2) at (-1.5,1) {};
	\node[] (a3) at (0,1)  {};
	\node[] (a4) at (1.5,1)  {};
	\node[] (a5) at (3,1)  {};
	
	\node[] (b1) at (-3,3) {};
	\node[] (b2) at (-1.5,3) {};
	\node[] (b3) at (0,3)  {};
	\node[] (b4) at (1.5,3)  {};
	\node[] (b5) at (3,3)  {};
	
	\node[] (t) at (0,5) {};
	
	\draw[edges] (s) -- (a1);
	\draw[edges] (s) -- (a2);
	\draw[edges] (s) -- (a3);
	\draw[edges] (s) -- (a4);
	\draw[edges] (s) -- (a5);
	
	\draw[edges] (a1) -- (b1);
	\draw[edges] (a1) -- (b2);
	\draw[edges] (a2) -- (b2);
	\draw[edges] (a2) -- (b3);
	\draw[edges] (a3) -- (b3);
	\draw[edges] (a3) -- (b4);
	\draw[edges] (a4) -- (b4);
	\draw[edges] (a5) -- (b5);
	
	\draw[edges] (b1) -- (t);
	\draw[edges] (b2) -- (t);
	\draw[edges] (b3) -- (t);
	\draw[edges] (b4) -- (t);
	\draw[edges] (b5) -- (t);
	\end{tikzpicture}\caption{A poset whose dual Chow polynomial has both plus and minus signs}\label{fig:counterex-positivity}
\end{figure}
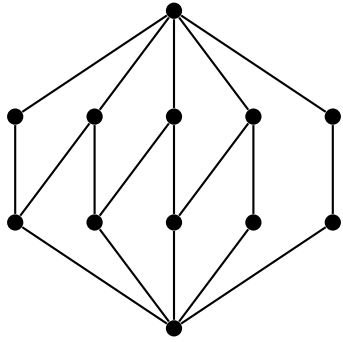
\end{example}

The preceding example hints at the requirement we need to impose in order to guarantee a ``sign-definiteness'' for the coefficients in $\H^*_P(x)$.
By \Cref{prop:degree-and-symmetry}(i), we know that the following holds for any weakly ranked poset $(P,\rho)$:
    \[ [x^{\rho_{st}-1}] \H^*_{st}(x) = (-1)^{\rho_{st}}[x^{\rho_{st}}]\chi^{\rev}_{st}(x) = (-1)^{\rho_{st}}\mu_{st}. \]
The M\"obius function \emph{alternates in sign} if it satisfies the condition 
\begin{equation}
\label{eq:moebius-alternates-sign}
 (-1)^{\rho_{st}}\mu_{st}\geq 0 \quad \text{for all intervals } [s,t]\,.
\end{equation}
For several important classes of posets the M\"obius function satisfies this condition \eqref{eq:moebius-alternates-sign},
but if that condition fails for some interval, the dual Chow polynomial might have some stray sign behavior. As the following theorem shows, under that assumption, we furthermore get a family of unimodal polynomials.

\begin{theorem}\label{thm:mobius-sign-unimodality}    
Let $(P,\rho)$ be a weakly ranked poset for which the M\"obius function alternates in sign: for every interval $[s,t]$ one has $(-1)^{\rho_{st}}\mu_{st}\geq 0$. 
Then, for each interval $[s,t]$, the polynomial $\H^*_{st}(x)$ has nonnegative and unimodal coefficients.
\end{theorem}

\begin{proof}
    This is a direct consequence of Corollary~\ref{cor:dual-sign-twist-unimodality}, because $f^*=\mu^{\sgn}$.
\end{proof}

As mentioned, the M\"obius function alternates in sign for several classes of posets. A key example are all Cohen--Macaulay posets; however, as we will show later, for that class of posets stronger phenomena than unimodality hold true. 

\begin{example}\label{ex:U34}
    Consider the poset $P$ depicted in Figure~\ref{fig:U34-haspic}. This is the truncation of the Boolean lattice of rank $4$ or, in other words, the lattice of flats of the uniform matroid $\U_{3,4}$. This poset is Cohen--Macaulay. A direct computation gives
    \[
        \H^*_{P}(x)=3x^2+11x+3,
    \]
    which is nonnegative and unimodal, as expected from the preceding result.

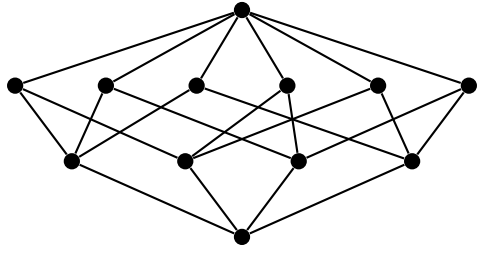
\begin{figure}[ht]
    \centering
    \begin{tikzpicture}
    [scale=0.5,auto=center,every node/.style={circle,scale=0.8, fill=black, inner sep=2.7pt}]
    \tikzstyle{edges} = [thick];

    \node[] (s)  at (0,-1) {};

    \node[] (a1) at (-4.5,1) {};
    \node[] (a2) at (-1.5,1) {};
    \node[] (a3) at (1.5,1)  {};
    \node[] (a4) at (4.5,1)  {};

    \node[] (b12) at (-6,3) {};
    \node[] (b13) at (-3.6,3) {};
    \node[] (b14) at (-1.2,3) {};
    \node[] (b23) at (1.2,3)  {};
    \node[] (b24) at (3.6,3)  {};
    \node[] (b34) at (6,3)  {};

    \node[] (t) at (0,5) {};

    \draw[edges] (s) -- (a1);
    \draw[edges] (s) -- (a2);
    \draw[edges] (s) -- (a3);
    \draw[edges] (s) -- (a4);

    \draw[edges] (a1) -- (b12);
    \draw[edges] (a1) -- (b13);
    \draw[edges] (a1) -- (b14);

    \draw[edges] (a2) -- (b12);
    \draw[edges] (a2) -- (b23);
    \draw[edges] (a2) -- (b24);

    \draw[edges] (a3) -- (b13);
    \draw[edges] (a3) -- (b23);
    \draw[edges] (a3) -- (b34);

    \draw[edges] (a4) -- (b14);
    \draw[edges] (a4) -- (b24);
    \draw[edges] (a4) -- (b34);

    \draw[edges] (b12) -- (t);
    \draw[edges] (b13) -- (t);
    \draw[edges] (b14) -- (t);
    \draw[edges] (b23) -- (t);
    \draw[edges] (b24) -- (t);
    \draw[edges] (b34) -- (t);
    \end{tikzpicture}
    \caption{The lattice of flats of the uniform matroid $\U_{3,4}$.}
    \label{fig:U34-haspic}
\end{figure}
\end{example}

\begin{example}
    Let $\Pi_{n}$ denote the partition lattice of the set $[n+1]$, ordered by refinement. We regard $\Pi_{n}$ as a graded lattice of rank $n$, with minimum element the partition into singletons and with maximum element the trivial partition. The first values of the dual Chow polynomial are as follows:
    \[
        \H^*_{\Pi_n}(x)=
        \begin{cases}
            1 & n=1,\\
            2x+2 & n=2,\\
            6x^2+18x+6 & n=3,\\
            24x^3+154x^2+154x+24 & n=4,\\
            120x^4+1440x^3+3000x^2+1440x+120 & n=5,\\
            720x^5+15098x^4+56118x^3+56118x^2+15098x+720 & n=6.
        \end{cases}
    \]
    These computations are consistent with the well-known identity
    $\mu_{\Pi_n}(\widehat{0},\widehat{1})=(-1)^n n!$.
    Indeed, by Proposition~\ref{prop:degree-and-symmetry}, the leading and constant
    coefficients of $\H^*_{\Pi_n}(x)$ are both equal to
    $(-1)^n\mu_{\Pi_n}(\widehat{0},\widehat{1})=n!$.
    Since partition lattices are Cohen--Macaulay, the displayed polynomials are
    nonnegative, unimodal, and symmetric, as predicted by
    Theorem~\ref{thm:mobius-sign-unimodality}.
\end{example}

\subsection{The dual augmented Chow functions}

The following provides a clean recursive way of computing the dual right augmented Chow polynomial of a poset $P$.

\begin{proposition}
    The dual right augmented Chow function satisfies the following two recursions:
    \begin{align*}
        \F^*_{st}(x) 
        &= -\sum_{s< w \leq t} (-1)^{\rho_{sw}}(1+x+\cdots + x^{\rho_{sw}})\cdot \F^*_{wt}(x) \,,\\
        &= -\sum_{s\leq w < t} (-1)^{\rho_{wt}}(1+x+\cdots + x^{\rho_{wt}})\cdot \F^*_{sw}(x) \,,
    \end{align*}
    with base cases $\F^*_{ss}(x) = 1$ for every $s\in P$.
    In other words, we have that for $s\leq t$ in P: 
    \[
        (F^{*})^{-1}_{st}(x) = (-1)^{\rho_{st}}\frac{x^{\rho_{st}+1}-1}{x-1}.
    \]
\end{proposition}

\begin{proof}
    By definition, we have
    \( \F^* = \H^* (f^*)^{\rev} = \H^* (\mu^{\rev})^{\sgn}.\)
    By inverting both sides, we obtain
    \[ (\F^*)^{-1} = ((\mu^{\rev})^{\sgn})^{-1} \left(\H^*\right)^{-1} = (\zeta^{\rev})^{\sgn} \left(-\overline{(\chi^{\rev})^{\sgn}}\right).\]
    This is the sign twist of the inverse computed for the untwisted kernel $\chi^{\rev}$, and hence
    \[
        (F^{*})^{-1}_{st}(x)
        =
        (-1)^{\rho_{st}}(1+x+\cdots+x^{\rho_{st}}).
    \]
    The second formula follows immediately with this inverse.
    The displayed recursion is the usual recursion for the inverse of an incidence function with diagonal entries equal to $1$.
\end{proof}

\begin{proposition}\label{prop:Hstar-from-Fstar}
    The dual Chow function and the dual right augmented Chow function are related by
    \[  \F^*_{st}(x) = \sum_{s\leq w\leq t} \H^*_{sw}(x) (-x)^{\rho_{wt}}\mu_{wt} \enspace \text{ and } \enspace  
    \H^*_{st}(x)
        =
        \sum_{s\leq w\leq t}
        \F^*_{sw}(x)\,(-x)^{\rho_{wt}}\,.
    \]
    Moreover, if $s<t$,
    \[
        x \H^*_{st}(x)
        =
        \sum_{s\leq w\leq t}(-1)^{\rho_{wt}}\F^*_{sw}(x).
    \]
\end{proposition}

\begin{proof}
    Since we are working with the characteristic kernel, the dual right KLS function is $f^*=\mu^{\sgn}$. Hence, by definition of the dual right augmented Chow function,
    \[ 
        \F^*=\H^*(f^*)^{\rev}=\H^*(\mu^{\rev})^{\sgn},
    \]
    which gives the first formula in the statement.
    Now, multiplying this one on the right by $((\mu^{\rev})^{\sgn})^{-1}=(\upzeta^{\rev})^{\sgn}$ gives
    \[
        \H^*=\F^*(\upzeta^{\rev})^{\sgn},
    \]
    and the displayed interval formula follows from the fact that
    $((\upzeta^{\rev})^{\sgn})_{wt}(x)=(-1)^{\rho_{wt}}x^{\rho_{wt}}$.
    
    For the second identity, \Cref{it:chow-symmetry} implies that
    $(\H^*)^{\rev}_{st}=x\H^*_{st}$ whenever $s<t$. Recall that $\F^*$ is invariant under the rev involution. Applying the rev involution to
    $\H^*=\F^*(\upzeta^{\rev})^{\sgn}$ gives
    \[
        (\H^*)^{\rev}
        =
        \F^*\upzeta^{\sgn}.
    \]
    Evaluating this identity on an interval $[s,t]$ with $s<t$ gives
    \[
        x\H^*_{st}(x)=
        \sum_{s\leq w\leq t}
        (-1)^{\rho_{wt}}\F^*_{sw}(x).\qedhere
    \]
\end{proof}

\subsection{Dual Chow polynomials from the extended \texorpdfstring{$\aaa\bbb$}{ab}-index}

We now show that the dual Chow function can also be recovered by a specialization of the extended $\aaa\bbb$-index introduced in Section~\ref{subsec:prelim-ab-index}.
More specifically, the related variant that we use for the dual right augmented Chow function is
\[
    \expsib_{st}(y,\aaa,\bbb)
    := \omega\bigl(\Psi_{st}(\aaa,\bbb) \bbb \bigr)
\]
for every interval $[s,t]$ with $s<t$ and $\expsib_{ss}(y,\aaa,\bbb)=1$.
Note that $\expsib$ is not obtained by simply appending a final $\bbb$ to the Poincar\'e-extended chain weight; the extra $\bbb$ is appended \emph{before} applying the $\omega$-transformation.

\begin{theorem}\label{thm:expsi-to-dualchow}
    Let $(P,\rho)$ be a weakly ranked poset, and assume that every interval of $P$ is graded with respect to $\rho$. For every $s\leq t$ one has
    \begin{align*}
        (1-x)^{\rho_{st}}\H^*_{st}(x) &= \expsitilde_{st}(-x,x,1)
        \enspace \text{and} \\
        (1-x)^{\rho_{st}}\F^*_{st}(x) &= \expsib_{st}(-x,x,1).
    \end{align*}
    In particular, if $P$ is bounded of rank $r$, then
    \begin{align*}
        \H^*_P(x)&=(1-x)^{-r}\expsitilde_P(-x,x,1)\enspace \text{ and }\\
        \F^*_P(x)&=(1-x)^{-r}\expsib_P(-x,x,1).
    \end{align*}
\end{theorem}

\begin{proof}
    We first prove the identity for $\H^*$ using the expansion of $\expsitilde$ given in \Cref{thm:expsi-via-poincare}.
    By Lemma~\ref{lem:inverses-as-sums-over-chains}, applied to the definition of dual Chow function, we obtain:
    \[
        \H^*_{st}(x)
        =
        \sum_{s=c_0<c_1<\cdots<c_m=t}
        (x-1)^{-m}(-1)^{\rho_{st}}\prod_{i=1}^{m}\chi^{\rev}_{c_{i-1}c_i}(x).
    \]
    On the other hand, after specializing $y=-x$ in \Cref{thm:expsi-via-poincare}, the Poincar\'e factor attached to the same chain becomes
    \[
        \prod_{i=1}^{m}\Poin_{c_{i-1}c_i}(-x)
        =
        \prod_{i=1}^{m}\chi^{\rev}_{c_{i-1}c_i}(x).
    \]
    Moreover, after applying $\iota$, the word attached to this chain evaluates at $\aaa=x$ and $\bbb=1$ to~$(x-1)^{\rho_{st}-m}$. Therefore
    \[
        \expsitilde_{st}(-x,x,1)
        =
        (1-x)^{\rho_{st}}\H^*_{st}(x).
    \]
    We now prove the augmented identity. We proceed by induction on $\rho_{st}$. The statement is clear when $s=t$. Assume $s<t$. 
    By \cite[Lemma~3.5]{hoster-stump-vecchi}, we have $\omega( (\aaa-\bbb)^k \bbb) = (\aaa-\bbb)^k \cdot (\bbb - (-y)^{k+1} \aaa)$, and obtain
    \begin{align*}
        \omega( \Psi_{st} (\aaa,\bbb) \bbb )
        &= \sum_{ s = c_0 < \dots < c_{m} = t }
        \prod_{i=0}^{m}(\aaa-\bbb)^{\rho_{c_{i-1} c_{i}} - 1 }
        (\bbb - (-y)^{\rho_{c_{i-1} c_{i}}} \aaa) \,. 
    \intertext{
    Splitting a chain in $[s,t]$ according to its first element $c_1=w>s$, we obtain}
        \expsib_{st}(-x,x,1)
        &=
        -\sum_{s<w\leq t}
        (x-1)^{\rho_{sw}}
        (1+x+\cdots+x^{\rho_{sw}})
        \expsib_{wt}(-x,x,1) \,,
    \intertext{
    which is, by induction,}
        &= -\sum_{s<w\leq t}
        (x-1)^{\rho_{sw}}(1-x)^{\rho_{wt}}
        (1+x+\cdots+x^{\rho_{sw}})\F^*_{wt}(x) \,.
    \intertext{
    Since $(1-x)^{\rho_{wt}}=(-1)^{\rho_{wt}}(x-1)^{\rho_{wt}}$, and since $\rho_{st}=\rho_{sw}+\rho_{wt}$, this becomes}
        &= (1-x)^{\rho_{st}}
        \left(-\sum_{s<w\leq t}
        (-1)^{\rho_{sw}}(1+x+\cdots+x^{\rho_{sw}})\F^*_{wt}(x)
        \right) \,.
    \end{align*}
    The recursion for $\F^*$ proved in the previous subsection identifies the expression in parentheses with $\F^*_{st}(x)$. Hence
    \[
        \expsib_{st}(-x,x,1)=(1-x)^{\rho_{st}}\F^*_{st}(x),
    \]
    as desired.
\end{proof}

The function $\expsitilde$ can be obtained from both $\exapsi$ and $\expsib$ in similar ways.
Define $\iota_R$ as the $\mathbb{Z}[y]$-linear transformation that deletes the right-most letter in every $\aaa\bbb$-word.
Then,
\(
    \iota_R \circ \omega = (1+y) \omega\circ \iota_R
\)
and we obtain the following schematic picture:
\begin{figure}[ht]
    \centering
	\begin{tikzpicture}
        \node (exapsi) at (0,0)
            {$\exapsi = \omega(\aaa\Psi)$};
            
        \node (expsib) at (6,0)
            {$\expsib = \omega(\Psi\bbb)$};

        \node (iotaexapsi) at (3,-1.5)
            {$\expsitilde = (1+y)\omega(\Psi)$};

        \draw[bend right=0, |->]
            (exapsi.south) to node[midway, left] {$\iota \ $} (iotaexapsi.north west);

        \draw[bend left=0, |->]
            (expsib.south) to node[midway, right] {$\ \iota_R$} (iotaexapsi.north east);
    \end{tikzpicture}
    \caption{Three extensions of $\aaa\bbb$-index variations.}\label{fig:extended-ab}
\end{figure}

\subsection{Gamma-positivity}

For a subset $S\subseteq[r-1]$, let $S^c=[r-1]\setminus S$. We shall say that $S$ is stable if it contains no two consecutive integers.

\begin{theorem}\label{thm:gamma exp}
    Let $P$ be a bounded weakly ranked poset of rank $r$. The polynomials $\H^*_P(x)$ and $\F^*_P(x)$ can be written as follows.
    \[
        \H^*_P(x)
        =
        \sum_{\substack{S\subseteq[r-1]\\ r-1\notin S \\ S \text{ stable}}}
        \beta_P(S^c)\,
        x^{|S|}(1+x)^{r-1-2|S|} \,,
    \]
    and
    \[
        \F^*_P(x)
        =
        \sum_{\substack{S\subseteq[r-1]\\ S\text{ stable}}}
        \beta_P(S^c)\,
        x^{|S|}(1+x)^{r-2|S|}.
    \]
    In particular, if $P$ has nonnegative flag $h$-vector, then $\H^*_P(x)$ and $\F^*_P(x)$ are $\gamma$-positive. This applies, in particular, to Cohen--Macaulay posets.
\end{theorem}

\begin{proof}
    Let $\omega_{\operatorname{ev}}$ denote the specialization of $\omega$ at $y=-x$, $\aaa=x$, and $\bbb=1$. For $S\subseteq[r-1]$, let $m_S=m_1\cdots m_{r-1}$ be the word such that $m_i=\bbb$ if $i\in S$, and $m_i=\aaa$ otherwise. With the usual notation for the flag $h$-vector, we have
    \[
        \Psi_P(\aaa,\bbb)
        =
        \sum_{S\subseteq[r-1]}\beta_P(S)m_S.
    \]
    Fix $S\subseteq[r-1]$, and put $T=S^c$. The set $T$ records the positions of the letters $\aaa$ in $m_S$. Under $\omega_{\operatorname{ev}}$,
    every occurrence of $\aaa\bbb$ contributes a factor $x(1-x)^2$, every remaining $\bbb$ contributes $(1-x)(1+x)$, and
    every remaining $\aaa$ contributes $0$.
    Thus $\omega_{\operatorname{ev}}(m_S)$ is nonzero exactly when $T$ is stable and $r-1\notin T$. In that case,
    \[
        \omega_{\operatorname{ev}}(m_S)
        =
        (1-x)^{r-1}x^{|T|}(1+x)^{r-1-2|T|}.
    \]
    Since $\expsitilde_P=(1+y)\omega(\Psi_P)$,
    Theorem~\ref{thm:expsi-to-dualchow} gives
    \[
        (1-x)^r\H^*_P(x)
        =
        (1-x)^r
        \sum_{\substack{T\subseteq[r-2]\\ T\text{ stable}}}
        \beta_P(T^c)x^{|T|}(1+x)^{r-1-2|T|}.
    \]
    Dividing by $(1-x)^r$ gives the stated formula for $\H^*_P(x)$.

    The argument for the augmented polynomial is the same, except that now
    \[
        \Psi_P(\aaa,\bbb)\bbb
        =
        \sum_{S\subseteq[r-1]}\beta_P(S)m_S\bbb.
    \]
    The final appended $\bbb$ removes the condition $r-1\notin T$. Hence a term contributes exactly when~$T=S^c$ is stable, and then
    \[
        \omega_{\operatorname{ev}}(m_S\bbb)
        =
        (1-x)^r x^{|T|}(1+x)^{r-2|T|}.
    \]
    Using again Theorem~\ref{thm:expsi-to-dualchow} yields the claimed $\gamma$-expansion of $\F^*_P(x)$.
    If the flag $h$-vector of $P$ is nonnegative, all coefficients in these $\gamma$-expansions are nonnegative, as desired.
\end{proof}

\begin{remark}
    If the flag $h$-vector of the poset $P$ is symmetric, meaning that
    \(
        \beta_P(S) = \beta_P(S^c)
    \),
    then the dual augmented Chow polynomial $\F^*_P$ equals the augmented Chow polynomial $\G_P$.
    This in particular holds for Eulerian posets, thanks to the Dehn--Sommerville relations (see \cite[Chapter~3]{stanley-ec1}). Note, however, that the equality $F^*_P(x) = G_P(x)$ for Eulerian posets (under the characteristic kernel) does not assume (nor it implies) any symmetry on the characteristic function of an Eulerian poset $P$ nor their intervals.
\end{remark}

\begin{example}
    Consider the poset $P$ depicted in Figure~\ref{fig:mobius-sign-not-gamma-positive}. This poset satisfies the condition in Theorem~\ref{thm:sign-twist-unimodality}, yet its flag $h$-vector contains negative entries. Moreover, its dual Chow polynomial is
        \[ \H^*_P(x) = 1 + x + x^2 + x^3,\]
    which is unimodal but is not $\gamma$-positive.
    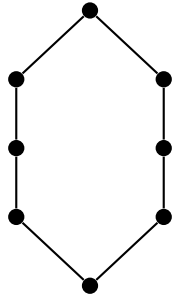
\begin{figure}[ht]
    \centering
    \begin{tikzpicture}
    [scale=0.65,auto=center,every node/.style={circle,scale=0.8, fill=black, inner sep=2.7pt}]
    \tikzstyle{edges} = [thick];

    \node[] (s)  at (0,0) {};

    \node[] (a1) at (-1.5,1.4) {};
    \node[] (a2) at (1.5,1.4) {};

    \node[] (b1) at (-1.5,2.8) {};
    \node[] (b2) at (1.5,2.8) {};

    \node[] (c1) at (-1.5,4.2) {};
    \node[] (c2) at (1.5,4.2) {};

    \node[] (t)  at (0,5.6) {};

    \draw[edges] (s) -- (a1);
    \draw[edges] (s) -- (a2);

    \draw[edges] (a1) -- (b1);
    \draw[edges] (a2) -- (b2);

    \draw[edges] (b1) -- (c1);
    \draw[edges] (b2) -- (c2);

    \draw[edges] (c1) -- (t);
    \draw[edges] (c2) -- (t);

    \end{tikzpicture}
    \caption{A rank $4$ poset satisfying the M\"obius sign condition, but whose dual Chow polynomial is not $\gamma$-positive.}
    \label{fig:mobius-sign-not-gamma-positive}
\end{figure}
\end{example}

\begin{remark}[Dual Chow polynomials of partitions lattices of types $\mathrm{A}$ and $\mathrm{B}$]
    Using \Cref{thm:gamma exp} and following \cite{stump}, we can the express dual Chow polynomials of the set partition lattice $\Pi_n$ similar to \cite[Theorem~3.1]{stump} (using weak ascents instead of descents).
    For the set partition lattice of type $\mathrm{B}$, that is the lattice of intersections of the Coxeter arrangement of type $\mathrm{B}$, we can modify \cite[Theorem~B]{degen-henetmayr-misinova-pielasa-rieg} in a similar manner to.
\end{remark}

\section{Dual Chow polynomials under poset operations}

\subsection{Ordinal sums}\label{subsec:ordinal-sums}

Recall that, if $P$ and $Q$ are two posets, their \emph{ordinal sum} $P\oplus Q$ is the poset on the disjoint union $P\sqcup Q$ that preserves the order relations of $P$ and $Q$, and imposes $s\leq t$ for every $s\in P$ and $t\in Q$. If $P$ and $Q$ are graded bounded posets, we define their \emph{join} $P*Q$ by identifying $\widehat{1}_P$ with $\widehat{0}_Q$. Equivalently,
\[
    P*Q=P\oplus\bigl(Q\setminus\{\widehat{0}_Q\}\bigr)
    =
    \bigl(P\setminus\{\widehat{1}_P\}\bigr)\oplus Q.
\]
The rank of $P*Q$ is the sum of the ranks of $P$ and $Q$. We also write $\aug(P)$ for the poset obtained from $P$ by adding a new minimum element.

We shall repeatedly use the following recursive form of the definition of $\H^*$. Since $\H^*=-(\overline{(\chi^{\rev})^{\sgn}})^{-1}$, for every $s<t$ one has
\begin{equation}\label{eq:dual-chow-recursion}
    \H^*_{st}(x)
    =
    \sum_{s<w\leq t}
    \overline{(\chi^{\rev})^{\sgn}}_{sw}(x)\,\H^*_{wt}(x).
\end{equation}

We start with a simple observation about the characteristic function of a join. If $s\in P\setminus\{\widehat{1}_P\}$ and $t\in Q$, then, in the poset $P*Q$,
\begin{equation}\label{eq:chi-rev-join}
    \chi^{\rev}_{st}(x)
    =
    \chi^{\rev}_{s,\widehat{1}_P}(x).
\end{equation}
Indeed, the M\"obius function of $P*Q$ satisfies $\mu_{su}=0$ for every $u\in Q\setminus\{\widehat{0}_Q\}$, and therefore
\[
    \chi^{\rev}_{st}(x)
    =
    \sum_{s\leq u\leq t} x^{\rho_{su}}\mu_{su}
    =
    \sum_{s\leq u\leq \widehat{1}_P} x^{\rho_{su}}\mu_{su}
    =
    \chi^{\rev}_{s,\widehat{1}_P}(x).
\]

\begin{lemma}\label{lem:dual-chow-augmentation}
    Let $P$ be a graded bounded poset. Then
    \[
        \H^*_{\aug(P)}(x)
        =
        \sum_{w\in P}(-1)^{\rho(w)}\H^*_{w,\widehat{1}_P}(x).
    \]
\end{lemma}

\begin{proof}
    Let $\widehat{0}$ be the new minimum element of $\aug(P)$. Since $\aug(P)=C_2*P$, equation~\eqref{eq:chi-rev-join} gives
    \[
        \chi^{\rev}_{\widehat{0},w}(x)=\chi^{\rev}_{C_2}(x)=1-x
    \]
    for every $w\in P$. Thus
    \[
        \overline{(\chi^{\rev})^{\sgn}}_{\widehat{0},w}(x)=(-1)^{\rho_{\widehat{0}_P,w}}
    \]
    for every $w\in P$. Applying \eqref{eq:dual-chow-recursion} to the interval $[\widehat{0},\widehat{1}_P]$ in $\aug(P)$ gives the desired identity.
\end{proof}

The following can be interpreted as a dual version of part of a result appearing in \cite[Proposition~4.5]{ferroni-matherne-vecchi}.

\begin{proposition}\label{prop:dual-chow-join}
    Let $P$ and $Q$ be graded bounded posets, and assume that $P$ has positive rank. Then
    \[
        \H^*_{P*Q}(x)
        =
        \H^*_P(x)\,\H^*_{\aug(Q)}(x).
    \]
\end{proposition}

\begin{proof}
    We argue by induction on the rank of $P$. If $P$ has rank $1$, then $P=C_2$, and $P*Q=\aug(Q)$. Since $\H^*_{C_2}(x)=1$, the identity follows.

    Assume now that the rank of $P$ is greater than $1$. We view $P*Q$ as the poset obtained by identifying $\widehat{1}_P$ with $\widehat{0}_Q$. Applying \eqref{eq:dual-chow-recursion} to $P*Q$, and splitting the sum according to whether the first element after $\widehat{0}_P$ lies in $P\setminus\{\widehat{0}_P,\widehat{1}_P\}$ or in $Q$, gives
    \begin{align*}
        \H^*_{P*Q}(x)
        &=
        \sum_{\widehat{0}_P<w<\widehat{1}_P}
        \overline{(\chi^{\rev})^{\sgn}}_{\widehat{0}_P,w}(x)
        \H^*_{[w,\widehat{1}_P]*Q}(x) \\
        &\quad+
        \overline{(\chi^{\rev})^{\sgn}}_{P}(x)
        \sum_{w\in Q}(-1)^{\rho_{\widehat{0}_Q,w}}\H^*_{w,\widehat{1}_Q}(x).
    \end{align*}
    By induction, $\H^*_{[w,\widehat{1}_P]*Q}(x)=\H^*_{w,\widehat{1}_P}(x)\H^*_{\aug(Q)}(x)$. Also, by Lemma~\ref{lem:dual-chow-augmentation},
    \[
        \sum_{w\in Q}(-1)^{\rho_{\widehat{0}_Q,w}}\H^*_{w,\widehat{1}_Q}(x)
        =
        \H^*_{\aug(Q)}(x).
    \]
    Hence
    \begin{align*}
        \H^*_{P*Q}(x)
        &=
        \H^*_{\aug(Q)}(x)
        \sum_{\widehat{0}_P<w<\widehat{1}_P}
        \overline{(\chi^{\rev})^{\sgn}}_{\widehat{0}_P,w}(x)
        \H^*_{w,\widehat{1}_P}(x)\\
        &\quad
        +\overline{(\chi^{\rev})^{\sgn}}_{P}(x)\H^*_{\aug(Q)}(x).
    \end{align*}
    Finally, applying \eqref{eq:dual-chow-recursion} inside $P$ gives
    \[
        \H^*_P(x)
        =
        \sum_{\widehat{0}_P<w<\widehat{1}_P}
        \overline{(\chi^{\rev})^{\sgn}}_{\widehat{0}_P,w}(x)
        \H^*_{w,\widehat{1}_P}(x)
        +
        \overline{(\chi^{\rev})^{\sgn}}_{P}(x).
    \]
    Substituting this into the previous expression proves the result.
\end{proof}

\begin{remark}
    In \cite[Proposition~4.5]{ferroni-matherne-vecchi} the authors proved that $\H_{P*Q}(x) = \H_P(x)\cdot G_Q(x)$. The analogue in our dual setting would be $\H^*_{P*Q}(x) = \H^*_P(x)\cdot \F^*_Q(x)$, but that does not hold in general, simply because $\H^*_{\aug(Q)}(x) \neq \F^*_Q(x)$.
\end{remark}

\subsection{Poset duality}\label{subsec:poset-duality}

We define the \emph{augmentation from the top} of a bounded poset $P$ to be the poset obtained by adding a new maximum element; we denote it by $\aug^*(P)$. Thus
\[
    \aug^*(P)=P*C_2.
\]

\begin{corollary}\label{cor:dual-chow-top-augmentation}
    Let $P$ be a graded bounded poset of positive rank. Then
    \[
        \H^*_{\aug^*(P)}(x)=0.
    \]
\end{corollary}

\begin{proof}
    By Proposition~\ref{prop:dual-chow-join}, we have
    \(
        \H^*_{\aug^*(P)}(x)
        =
        \H^*_{P*C_2}(x)
        =
        \H^*_P(x)\H^*_{\aug(C_2)}(x).
    \)
    Note that Lemma~\ref{lem:dual-chow-augmentation} gives
    \(
        \H^*_{\aug(C_2)}(x)
        =
        \H^*_{C_2}(x)-1=0,
    \)
    because $\H^*_{C_2}(x)=1$.
\end{proof}

A result by Br\"and\'en and Vecchi in \cite[Corollary~4.2]{branden-vecchi-uniform} guarantees that left augmented Chow polynomials are invariant under poset duality. The dual analogue of that statement is also true.

\begin{proposition}\label{prop:Fstar-poset-duality}
    Let $P$ be a graded bounded poset, and let $P^*$ be its dual poset. Then
    \[
        \F^*_P(x)=\F^*_{P^*}(x).
    \]
\end{proposition}

\begin{proof}
    Recall that, for the characteristic kernel, the dual right augmented Chow function satisfies
    \[
        (\F^*)^{-1}_{st}(x)=(-1)^{\rho_{st}}(1+x+\cdots+x^{\rho_{st}}).
    \]
    By Lemma~\ref{lem:inverses-as-sums-over-chains}, $\F^*_{st}(x)$ can therefore be written as a signed sum over chains from $s$ to $t$, where the contribution of a chain depends only on the ranks of its successive intervals. Reversing chains gives a weight-preserving bijection between chains in $P$ from $\widehat{0}$ to $\widehat{1}$ and chains in $P^*$ from $\widehat{0}_{P^*}$ to $\widehat{1}_{P^*}$. Hence the two polynomials agree.
\end{proof}

Ferroni, Matherne, and Vecchi \cite[Corollary~4.6]{ferroni-matherne-vecchi} proved that all left augmented Chow polynomials are themselves the Chow polynomial of suitable constructed posets. Conversely, Br\"and\'en and Vecchi proved that all Chow polynomials are left augmented Chow polynomials of other posets \cite[Theorem~1.2]{branden-vecchi-uniform}. The following provides a counterpart of Br\"and\'en and Vecchi's result.

\begin{proposition}\label{prop:Fstar-top-augmentation}
    Let $P$ be a non-trivial graded bounded poset. Then
    \[
        \H^*_P(x) = \frac{1}{x}\,\F^*_{\aug^*(P)}(x).
    \]
\end{proposition}

\begin{proof}
    Let $Q=\aug^*(P)$, and let $\widehat{1}_Q$ denote the new maximum element. By Corollary~\ref{cor:dual-chow-top-augmentation}, we have $\H^*_Q(x)=0$. Applying Proposition~\ref{prop:Hstar-from-Fstar} to the interval $[\widehat{0}_P,\widehat{1}_Q]$ in $Q$ gives
    \[
        0=x\H^*_Q(x)
        =
        \sum_{\widehat{0}_P\leq w\leq \widehat{1}_Q}
        (-1)^{\rho_{w,\widehat{1}_Q}}\F^*_{\widehat{0}_P,w}(x).
    \]
    The summand corresponding to $w=\widehat{1}_Q$ is $\F^*_{\aug^*(P)}(x)$. For every $w\in P$, the interval $[\widehat{0}_P,w]$ is the same whether computed in $P$ or in $Q$. Hence
    \[
        \F^*_{\aug^*(P)}(x)
        =
        \sum_{w\in P}(-1)^{\rho_{w,\widehat{1}_P}}\F^*_{\widehat{0}_P,w}(x).
    \]
    Applying Proposition~\ref{prop:Hstar-from-Fstar} once more, now to the interval $[\widehat{0}_P,\widehat{1}_P]$ in $P$, gives
    \[
        \sum_{w\in P}(-1)^{\rho_{w,\widehat{1}_P}}\F^*_{\widehat{0}_P,w}(x)=x\H^*_P(x),
    \]
    because $P$ has positive rank. This proves the desired identity.
\end{proof}

\subsection{Cartesian products}

We are now interested in the dual Chow polynomial of the product of two posets. We recall that the product of two weakly ranked bounded posets $P$ and $Q$ is the weakly ranked bounded poset denoted as $P \times Q$ which is the poset over the set $P \times Q$ with relations $(s,t) \leq (s',t') $ if and only if $s \leq s'$ in $P$ and $t \leq t'$ in $Q$, and whose weak rank function $\rho_{P\times Q}$ is given by $\rho_{P\times Q}(s,t) = \rho_P(s) + \rho_Q(t)$ (note that this extends to all intervals in $P\times Q$ in the obvious way).

For a poset $P$ and an element $s\in P$, let us denote $P_{\leq s}:= \{x\in P : x\leq s\}$. Define $P_{\geq s}$ analogously. Note that when $P$ is bounded, one has $P_{\leq s} = [\widehat{0},s]$ and $P_{\geq s} = [s,\widehat{1}]$. 

\begin{proposition}
    Let $P$ and $Q$ be two bounded posets. The dual Chow polynomial of their Cartesian product can be computed as follows:
    \[ \H^*_{P\times Q}(x) = \H^*_P(x)\, \H^*_Q(x) + x \sum_{\substack{s\in P\\ s\neq \widehat{1}_P}} \sum_{\substack{t\in Q\\ t\neq \widehat{1}_Q}} \H^*_{P_{\leq s} \times Q_{\leq t}}(x) \H^*_{P_{\geq s}}(x) \H^*_{Q_{\geq t}}(x).\]
\end{proposition}

\begin{proof}
    This follows directly from Pielasa's formula for Chow functions under Cartesian products \cite[Theorem~1.9]{pielasa}, after noticing that the sign-twisted reversed characteristic polynomial of a product of posets is the product of the sign-twisted reversed characteristic polynomials of the factors.
\end{proof}

\subsection{Larson's recursions in the dual case}

We now discuss the behavior of the dual Chow polynomial under truncation. Recall that, if $P$ is a graded bounded poset of positive rank, then $\trunc(P)$ denotes the poset obtained from $P$ by removing its coatoms. Equivalently, $\trunc(P)$ has the same elements as $P$ in ranks strictly smaller than $\rank(P)-1$, together with the same maximum element. Thus $\trunc(P)$ has rank $\rank(P)-1$.

We begin with a refinement at the level of the extended $\aaa\bbb$-index. For an interval $[s,t]$ of $P$, define
\[
    K_{st}(y,\aaa,\bbb)
    =
    \begin{cases}
        1 & \text{if } s=t,\\
        -\Poin_{st}(y)\,\bbb(\aaa-\bbb)^{\rho_{st}-1} & \text{if } s<t,
    \end{cases}
\]
and
\[
    M_{st}(y,\aaa,\bbb)
    =
    \begin{cases}
        1 & \text{if } s=t,\\
        \mu_{st}(-y)^{\rho_{st}-1}(1+y)\bbb(\aaa-\bbb)^{\rho_{st}-1} & \text{if } s<t.
    \end{cases}
\]
For a bounded interval $[s,t]$, we shall also write $K_{[s,t]}$ and $M_{[s,t]}$ when convenient.

\begin{proposition}\label{prop:extended-ab-truncation}
    Let $P$ be a graded bounded poset of rank at least $2$. Then
    \[
        \exapsi_{\trunc(P)}(y,\aaa,\bbb)(\aaa-\bbb)
        =
        (\exapsi\cdot M)_P(y,\aaa,\bbb).
    \]
    Moreover,
    \[
        \expsitilde_{\trunc(P)}(y,\aaa,\bbb)(\aaa-\bbb)
        =
        (\expsitilde\cdot M)_P(y,\aaa,\bbb)
        +(1-\bbb)\,\iota(M_P(y,\aaa,\bbb)).
    \]
\end{proposition}

\begin{proof}
    We first record the elementary recursion obtained by separating a chain according to its last element before the maximum. For every bounded graded poset $P$ of positive rank,
    \[
        \exapsi_P(y,\aaa,\bbb)
        =
        \sum_{w\neq\widehat{1}}
        \exapsi_{\widehat{0},w}(y,\aaa,\bbb)\,
        \bigl(-K_{w,\widehat{1}}(y,\aaa,\bbb)\bigr)
        +
        (\aaa-\bbb)^{\rank(P)}.
    \]
    Indeed, if a chain in $[\widehat{0},\widehat{1})$ is nonempty, let $w$ be its maximal element; the factor $-K_{w,\widehat{1}}$ records the last Poincar\'e factor, the letter $\bbb$ at the rank of $w$, and the remaining letters $\aaa-\bbb$. The last term corresponds to the empty chain.

    We shall use the standard identity
    \[
        \Poin_{\trunc([s,t])}(y)
        =
        \Poin_{st}(y)+\mu_{st}(-y)^{\rho_{st}-1}(1+y)
    \]
    for intervals of rank greater than $1$. It implies
    \[
        -K_{\trunc([s,t])}(y,\aaa,\bbb)(\aaa-\bbb)
        =
        -K_{st}(y,\aaa,\bbb)+M_{st}(y,\aaa,\bbb)
    \]
    whenever $\rho_{st}>1$. For intervals of rank $0$ or $1$, the right hand side is zero.

    Applying the preceding chain decomposition to $\trunc(P)$ and multiplying by $\aaa-\bbb$, we obtain
    \[
        \exapsi_{\trunc(P)}(y,\aaa,\bbb)(\aaa-\bbb)
        =
        \sum_{\substack{w\in P\\ \rho_{w,\widehat{1}}>1}}
        \exapsi_{\widehat{0},w}(y,\aaa,\bbb)
        \bigl(-K_{\trunc([w,\widehat{1}])}(y,\aaa,\bbb)\bigr)(\aaa-\bbb)
        +
        (\aaa-\bbb)^{\rank(P)}.
    \]
    Replacing the last factor by $-K_{w,\widehat{1}}+M_{w,\widehat{1}}$, and then using the chain decomposition of $\exapsi_P$, gives
    \[
        \exapsi_{\trunc(P)}(y,\aaa,\bbb)(\aaa-\bbb)
        =
        \sum_{w\in P}
        \exapsi_{\widehat{0},w}(y,\aaa,\bbb)M_{w,\widehat{1}}(y,\aaa,\bbb),
    \]
    which is the first identity.

    For the second identity, apply $\iota$ to the first one. If $w>\widehat{0}$, then the first letter belongs to the factor $\exapsi_{\widehat{0},w}$, and hence
    \[
        \iota\bigl(\exapsi_{\widehat{0},w}M_{w,\widehat{1}}\bigr)
        =
        \expsitilde_{\widehat{0},w}M_{w,\widehat{1}}.
    \]
    The only exceptional summand is $w=\widehat{0}$, where applying $\iota$ gives $\iota(M_P)$ instead of $M_P$. Thus
    \[
        \iota\bigl((\exapsi\cdot M)_P\bigr)
        =
        (\expsitilde\cdot M)_P+\iota(M_P)-M_P.
    \]
    Since $M_P=\bbb\,\iota(M_P)$ for positive rank, the desired correction term is
    \[
        \iota(M_P)-M_P=(1-\bbb)\iota(M_P).
    \]
    This proves the second identity.
\end{proof}

\begin{remark}
    Specializing Proposition~\ref{prop:extended-ab-truncation} at $y=-x$, $\aaa=1$, and $\bbb=x$, and using Theorem~\ref{thm:stump}, recovers Larson's recursions for the ordinary Chow and left augmented Chow polynomials:
    \[
        \G_{\trunc(P)}(x)=(\G\cdot\mu^{\rev})_P(x)
    \]
    and
    \[
        \H_{\trunc(P)}(x)
        =
        (\H\cdot\mu^{\rev})_P(x)
        +(1-x)\mu_Px^{\rank(P)-1}.
    \]
    After applying the rev involution and inverting the M\"obius function, these are equivalent to the form stated in \cite[Proposition~4.12]{ferroni-matherne-vecchi}.
\end{remark}

We now specialize instead at $y=-x$, $\aaa=x$, and $\bbb=1$. This produces the dual analogue of Larson's recursion (first proved by Larson in \cite[Corollary~3.5]{larson}). Define an incidence function $\widetilde{\mu}\in\mathcal{I}(P)$ by
\[
    \widetilde{\mu}_{st}(x)
    =
    \begin{cases}
        1 & \text{if } s=t,\\
        \mu_{st}(-x)^{\rho_{st}-1} & \text{if } s<t.
    \end{cases}
\]
Since $\widetilde{\mu}_{ss}=1$ for every $s$, this element is invertible; we denote its inverse by $\widetilde{\zeta}=\widetilde{\mu}^{-1}$.

\begin{proposition}\label{prop:dual-larson-convolution}
    Let $P$ be a graded bounded poset. Then
    \[
        (\H^*\cdot\widetilde{\mu})_P(x)
        =
        \begin{cases}
            1 & \text{if } \rank(P)=0,\\
            0 & \text{if } \rank(P)=1,\\
            -\H^*_{\trunc(P)}(x) & \text{if } \rank(P)>1.
        \end{cases}
    \]
\end{proposition}

\begin{proof}
    The cases $\rank(P)=0$ and $\rank(P)=1$ follow immediately from the definitions. Assume that $P$ has rank $r>1$. Specialize the second identity of Proposition~\ref{prop:extended-ab-truncation} at $y=-x$, $\aaa=x$, and $\bbb=1$. The correction term vanishes, because $1-\bbb=0$. By Theorem~\ref{thm:expsi-to-dualchow}, the left hand side becomes
    \[
        \expsitilde_{\trunc(P)}(-x,x,1)(x-1)
        =
        -(1-x)^r\H^*_{\trunc(P)}(x).
    \]
    On the right hand side, for an interval $[w,\widehat{1}]$ of positive rank $d$, we have
    \[
        M_{w,\widehat{1}}(-x,x,1)
        =
        \mu_{w,\widehat{1}}(-x)^{d-1}(1-x)^d.
    \]
    Thus the summand indexed by $w$ contributes
    \[
        (1-x)^{\rho_{\widehat{0},w}}\H^*_{\widehat{0},w}(x)\,
        \mu_{w,\widehat{1}}(-x)^{\rho_{w,\widehat{1}}-1}(1-x)^{\rho_{w,\widehat{1}}},
    \]
    and hence the entire right hand side is
    \[
        (1-x)^r(\H^*\cdot\widetilde{\mu})_P(x).
    \]
    Dividing by $(1-x)^r$ gives the stated identity.
\end{proof}

Inverting the preceding convolution identity gives a recursion for $\H^*_P$ in terms of truncations of lower intervals.

\begin{corollary}\label{cor:dual-larson-recursion}
    Let $P$ be a graded bounded poset. Then
    \[
        \H^*_P(x)
        =
        \widetilde{\zeta}_P(x)
        -
        \sum_{\substack{w\in P\\ \rho(w)>1}}
        \H^*_{\trunc([\widehat{0},w])}(x)\,
        \widetilde{\zeta}_{w,\widehat{1}}(x).
    \]
\end{corollary}

\begin{proof}
    Let $A\in\mathcal{I}(P)$ be the incidence function defined by
    \[
        A_{st}(x)
        =
        \begin{cases}
            1 & \text{if } \rho_{st}=0,\\
            0 & \text{if } \rho_{st}=1,\\
            -\H^*_{\trunc([s,t])}(x) & \text{if } \rho_{st}>1.
        \end{cases}
    \]
    Proposition~\ref{prop:dual-larson-convolution} says precisely that
    \[
        \H^*\cdot\widetilde{\mu}=A.
    \]
    Multiplying on the right by $\widetilde{\zeta}=\widetilde{\mu}^{-1}$ gives $\H^*=A\cdot\widetilde{\zeta}$. Evaluating this identity on the interval $[\widehat{0},\widehat{1}]$ gives the formula.
\end{proof}

\section{Dual Chow polynomials of matroids}\label{sec:matroids}

\subsection{Deletion formula for ab-indices of matroids}

We now prove a deletion formula for the dual Chow polynomial of a matroid. The main input is the usual description of the $\aaa\bbb$-index of a geometric lattice in terms of an $R$-labeling.

Let $\M$ be a matroid of rank $r$ on the ground set $E$. We write $\cL(\M)$ for its lattice of flats, and we abbreviate
\[
    \H^*_{\M}(x):=\H^*_{\mathcal{L}(\M)}(x).
\]
Fix a total order on $E$. For every cover $F\lessdot G$ in $\cL(\M)$, let $\lambda(F\lessdot G)$ be the minimum element of $G\setminus F$ with respect to this order. If $\cF=\{F_0\lessdot F_1\lessdot\cdots\lessdot F_r\}$ is a maximal chain in $\cL(\M)$, we write $\lambda_\cF=(\lambda_1,\ldots,\lambda_r)$, where $\lambda_j=\lambda(F_{j-1}\lessdot F_j)$. We also denote by $u_\cF(\aaa,\bbb)=u_1\cdots u_{r-1}$ the word defined by
\[
    u_j=
    \begin{cases}
        \aaa & \text{if } \lambda_j\leq \lambda_{j+1},\\
        \bbb & \text{if } \lambda_j> \lambda_{j+1},
    \end{cases}
    \qquad 1\leq j\leq r-1.
\]
That is, $u_\cF(\aaa,\bbb) = \operatorname{m}_D(\aaa,\bbb)$ for $D$ the descent set of $\lambda_\cF$, and we thus call $u_\cF(\aaa,\bbb)$ a descent word.
With this notation, the $\aaa\bbb$-index of~$\cL(\M)$ is
\[
    \Psi_{\cL(\M)}(\aaa,\bbb)=\sum_{\cF} u_\cF(\aaa,\bbb),
\]
where the sum ranges over all maximal chains of $\cL(\M)$.

Throughout the rest of the subsection, fix an element $i\in E$ which is not a coloop and that has no parallel elements (i.e., $\{i\}$ is a flat), and choose the total order above so that $i$ is maximal. We shall also assume that $\zero\cup\{i\}$ is a rank-one flat, which is automatic if $\M$ is simple.

\begin{lemma}\label{lem:chains-deletion}
    The maximal chains in $\cL(\M\setminus i)$ are in bijection with the maximal chains $\cF=\{F_0\lessdot\cdots\lessdot F_r\}$ in $\cL(\M)$ such that $F_{j+1}\setminus F_j\neq\{i\}$ for every $j$. Moreover, this bijection preserves the descent word associated to the $R$-labeling.
\end{lemma}

\begin{proof}
    We use the identity $\cL(\M\setminus i)=\{F\setminus\{i\}:F\in\cL(\M)\}$. If $\cF$ is a maximal chain in $\cL(\M)$ with no cover differing by the single element $i$, then the chain obtained by replacing every $F_j$ with $F_j\setminus\{i\}$ is a maximal chain in $\cL(\M\setminus i)$, because $i$ is not a coloop.

    Conversely, let $\cC=\{C_0\lessdot\cdots\lessdot C_r\}$ be a maximal chain in $\cL(\M\setminus i)$. Let $k$ be the largest index such that $C_j$ is a flat of $\M$ for every $j\leq k$. We define a chain in $\cL(\M)$ by keeping $C_j$ unchanged for $j\leq k$ and replacing $C_j$ by $C_j\cup\{i\}$ for $j>k$. Indeed, if $C_j$ is a flat of $\M\setminus i$, then its closure in $\M$ is either $C_j$ or $C_j\cup\{i\}$; once $i$ belongs to the closure of some $C_j$, it belongs to the closure of all later $C_\ell$. This gives a maximal chain in $\cL(\M)$ with no cover differing by exactly $i$.

    The two constructions are inverse to one another. Since $i$ is maximal in the chosen total order, adding or removing $i$ simultaneously from the two flats in a cover does not change the minimum element added in that cover. Hence the descent word is preserved.
\end{proof}

Define
\[
    \mathscr{S}_i=\{F\in\cL(\M):i\notin F,\;F\cup\{i\}\in\cL(\M)\},
    \qquad
    \underline{\mathscr{S}}_i=\mathscr{S}_i\setminus\{\zero\}.
\]
Every maximal chain in $\cL(\M)$ either corresponds to a chain in $\cL(\M\setminus i)$ by Lemma~\ref{lem:chains-deletion}, or contains a unique cover of the form $F\lessdot F\cup\{i\}$ with $F\in \mathscr{S}_i$.
In the latter case, the label of this cover is $i$, which is larger than all other labels.

\begin{theorem}
\label{thm:ab-deletion}
    The $\aaa\bbb$-index of $\cL(\M)$ satisfies the following recursion:
    \begin{align*}
        \Psi_{\cL(\M)}
        &= \Psi_{\cL(\M\setminus i)} 
            + \bbb \Psi_{\cL(\M/i)} 
            + \sum_{F\in \underline{\mathscr{S}}_i } 
                \Psi_{\cL(\M|F)}\, \aaa\bbb\, \Psi_{\cL(\M/(F\cup\{i\}))} \,.
    \end{align*}
\end{theorem}
\begin{proof}
    Let $\cF$ be a maximal chain with sequence $\lambda_{\cF}=(\lambda_1,\dots, \lambda_{r})$. Assume $r\geq 2$.
    If $i$ is not contained in $\lambda_\cF$, then $\cF$ contributes to the first summand. Otherwise $\cF$ contains a cover $F\lessdot F\cup\{i\}$ and
    \[
        \lambda_\cF = ( \underbrace{\lambda_1 , \ldots, \lambda_{k-1} }_{\lambda_{\cG_1}}, \ i\ , \underbrace{\lambda_{k+1}, \ldots , \lambda_r}_{\lambda_{\cG_2}} )
    \]
    with $( \lambda_1 , \ldots, \lambda_k)$ the sequence of the maximal chain $\cG_1 =\cF \cap [\zero, F]$ in $\cL(\M\setminus i)$,
    and $(\lambda_{k+2}, \ldots , \lambda_r )$ the sequence of the maximal chain $\cG_2$ in $\cL(\M/(F\cup\{i\}))$ corresponding to $\cF$ restricted to $[F\cup\{i\},\one]$.
    Notice that $\lambda_{\cG_1}$ might be empty depending on the position $k\in \{1,\dots , r-1\}$ of $i = \lambda_k$.
    It is only possible to have $\lambda_r=i$ if $E\setminus \{i\}$ is a flat of $\M$, which is equivalent to $i$ being a coloop, and thus excluded.
    Since $i$ is the maximal possible label (with respect to our fixed total order), we obtain
    \[
        u_{\cF} = 
        \begin{cases}
            \phantom{u_{\cG_1} \, \aaa} \bbb \, u_{\cG_2} & \text{if } k = 1 \\
            u_{\cG_1} \, \aaa \bbb \, u_{\cG_2} & \text{if } 1 < k < r \,.
        \end{cases}
    \]
    These cases contribute to the second and third summand in the claimed expansion of~$\Psi_{\cL(\M)}$.
\end{proof}

\begin{corollary}\label{thm:extended-ab-deletion}
    The extended $\aaa\bbb$-index and the modified extended $\aaa\bbb$-index satisfy the following recursions:
    \begin{align*}
        \exapsi_{\cL(\M)}
        &= \exapsi_{\cL(\M\setminus i)}
         + 
            \sum_{F\in \mathscr{S}_i }  
                \exapsi_{\cL(\M|F)}\,
                (\aaa\bbb + y\bbb\aaa)\,
                \expsitilde_{\cL(\M/(F\cup\{i\}))} 
    \intertext{and}
        \expsitilde_{\cL(\M)}
        &= \expsitilde_{\cL(\M\setminus i)} 
            + (\bbb+y\aaa)\, \expsitilde_{\cL(\M/i)} \\
        & \qquad + 
            \sum_{F\in \underline{\mathscr{S}}_i}
                \expsitilde_{\cL(\M|F)}\,
                (\aaa\bbb+y\bbb\aaa)\,
                \expsitilde_{\cL(\M/(F\cup\{i\}))} \,.
    \end{align*}
\end{corollary}

\begin{proof}
    Since $\exapsi_{\cL(\M)} = \omega(\aaa\Psi_{\cL(\M)})$, we prepend $\aaa$ to the identity in \Cref{thm:ab-deletion} and then apply $\omega$.
    Using $\omega(\aaa\bbb) = (1+y)(\aaa\bbb + y\bbb\aaa)$ and $(1+y)\omega(\Psi) = \expsitilde$, we prove the first identity
    \begin{align*}
        \exapsi_{\cL(\M)}
        &= \exapsi_{\cL(\M\setminus i)}
         + (\aaa\bbb+y\bbb\aaa) \expsitilde_{\cL(\M/i)} \\
        & \quad  + \sum_{F\in \underline{\mathscr{S}}_i } 
        \exapsi_{\cL(\M|F)}\,
        (\aaa\bbb + y\bbb\aaa)\,
        \expsitilde_{\cL(\M/(F\cup\{i\}))} \,.
    \end{align*}
    For the second identity, we apply $\iota$ to the first and obtain
    \begin{align*}
        \expsitilde_{\cL(\M)}
        &= \expsitilde_{\cL(\M\setminus i)} 
         + (\bbb+y\aaa) \expsitilde_{\cL(\M/i)} \\
        & \quad  + \sum_{F\in \underline{\mathscr{S}}_i } 
        \expsitilde_{\cL(\M|F)}\,
        (\aaa\bbb + y\bbb\aaa)\,
        \expsitilde_{\cL(\M/(F\cup\{i\}))} \,.\qedhere
    \end{align*}
\end{proof}

\begin{remark}
    Specializing the second identity in Theorem~\ref{thm:extended-ab-deletion} at $y=-x$, $\aaa=1$, and $\bbb=x$, and using Stump's formulas stated in Theorem~\ref{thm:stump}, one recovers the usual semi-small decomposition for the Chow polynomial (see \cite[Corollary~3.24]{ferroni-matherne-stevens-vecchi}).
\end{remark}

Reasoning in an entirely analogous way, one obtains the following counterpart of the above result.

Define $\exapsib_{st} := \omega(\aaa\Psi_{st}\bbb)$ for $s<t$ and $\exapsib_{ss}=1$, according to our notation \eqref{eq:ex-and-tilde}.

\begin{corollary}\label{thm:modified-extended-ab-deletion}
    The extensions $\exapsib_{\cL(\M)}$ and $\expsib_{\cL(\M)}$ satisfy the following recursions:
    \begin{align*}
        \exapsib_{\cL(\M)}
        &= \exapsib_{\cL(\M\setminus i)} 
            + \sum_{F\in \mathscr{S}_i } 
        (1+y) \exapsi_{\cL(\M|F)}\,
        (\aaa\bbb + y\bbb\aaa)\,
        \expsib_{\cL(\M/(F\cup\{i\}))} 
    \intertext{and}
        \expsib_{\cL(\M)}
        &= \expsib_{\cL(\M\setminus i)} 
            + (\bbb+y\aaa) \expsib_{\cL(\M/i)} \\
        & \qquad  + \sum_{F\in \underline{\mathscr{S}}_i } 
        \expsitilde_{\cL(\M|F)}\,
        (\aaa\bbb+y\bbb\aaa)\,
        \expsib_{\cL(\M/(F\cup\{i\}))}\,.
    \end{align*}
\end{corollary}

\subsection{The deletion formula for dual Chow polynomials of matroids}

We specialize $\expsitilde$ and $\expsib$ at $y=-x$, $\aaa=x$, and $\bbb=1$. By Theorem~\ref{thm:expsi-to-dualchow}, this gives the corresponding decomposition for the dual Chow polynomial. 

\begin{theorem}\label{thm:dual-chow-deletion}
    Let $\M$ be a matroid of rank $r$, and let $i\in E$ be an element that is not a coloop and has no parallel elements. Then the dual Chow polynomial satisfies:
    \[
        \H^*_{\M}(x)
        =
        \H^*_{\M\setminus i}(x)
        +(x+1)\H^*_{\M/i}(x)
        +x\sum_{F\in\underline{\mathscr{S}}_i}
        \H^*_{\M|F}(x)\,
        \H^*_{\M/(F\cup\{i\})}(x).
    \]
    Similarly, the right augmented dual Chow polynomial satisfies:
    \[
    \F^*_{\M}(x)= \F^*_{\M \setminus i}(x) + (x+1)\F^*_{\M /i}(x) +x\sum_{F \in \underline{\mathscr{S}}_i} \H^*_{\M|F}(x)\F^*_{\M / (F\cup \{i\})}(x).
\]
\end{theorem}

\begin{proof}
    Specialize the second identity of Theorem~\ref{thm:extended-ab-deletion} at $y=-x$, $\aaa=x$, and $\bbb=1$. For every minor $\N$ of $\M$, Theorem~\ref{thm:expsi-to-dualchow} gives $\expsitilde_{\cL(\N)}(-x,x,1)=(1-x)^{\rk(\N)}\H^*_{\N}(x)$. Since $i$ is not a coloop, $\rk(\M\setminus i)=r$ and $\rk(\M/i)=r-1$. Also, for $F\in\underline{\mathscr{S}}_i$, one has $\rk(\M|F)+\rk(\M/(F\cup\{i\}))=r-1$.

    Finally, under the same specialization, $\bbb+y\aaa$ becomes $1-x^2=(x+1)(1-x)$, whereas $\aaa\bbb+y\bbb\aaa$ becomes $x-x^2=x(1-x)$. Thus every term contains the common factor $(1-x)^r$, and after dividing by it we obtain the stated formula for $\H^*_{\M}(x)$.

    The proof for $\F^*_{\M}(x)$ is identical, using Theorem~\ref{thm:modified-extended-ab-deletion} instead of Theorem~\ref{thm:extended-ab-deletion}.
\end{proof}

\begin{remark}
    The preceding theorem also shows, via a straightforward induction, that $\H^*_{\M}(x)$ is $\gamma$-positive for every matroid $\M$. The reason is that all the summands on the right hand side are symmetric polynomials with the same center of symmetry, and one can inductively assume that they are $\gamma$-positive. 
\end{remark}

\subsection{Dual Chow polynomials of uniform matroids}

For each integer $n\geq 1$, the $n$-th Eulerian polynomial $A_n(x)$ and the $n$-th binomial Eulerian polynomial $\widetilde{A}_n(x)$ are defined by 
\[
    A_n(x) = \sum_{w\in \mathfrak{S}_n} x^{\des(w)}\,,
    \quad \text{and} \quad
    \widetilde{A}_n(x) = 1 + x\, \sum_{k=1}^n \binom{n}{k} A_k(x) \,.
\]
We further define $A_0(x) = \widetilde{A}_0(x) = 1$. For $\kappa = \chi$, the Chow polynomial and the augmented Chow polynomial of the Boolean matroid $\U_{r,r}$ are given by
\[
    \H_{\U_{r,r}}(x) = A_r(x)
    \quad \text{and} \quad
    \G_{\U_{r,r}}(x) = \widetilde{A}_r(x) \,,
\]
see, e.g., \cite{hameister-rao-simpson}.

\begin{proposition}
    For the Boolean matroid, the dual Chow polynomial equals the Chow polynomial. That is, 
    \(
        \H^*_{\U_{r,r}}(x) = A_r(x)
        \text{ and }
        \F^*_{\U_{r,r}}(x) = \widetilde{A}_r(x) \,.
    \)
\end{proposition}
\begin{proof}
    The flag $h$-vector of $\U_{r,r}$ satisfies $\beta_{\cL(\U_{r,r})}(S) = \beta_{\cL(\U_{r,r})}(S^c)$ for all $S\subseteq [r-1]$.
    In addition, $\cL(\U_{r,r})$ is self-dual. 
    Thus, the $\gamma$-expansions stated in \Cref{thm:gamma exp} for $\H^*_{\U_{r,r}}$ and $\F^*_{\U_{r,r}}$ coincide with the $\gamma$-expansions for $\H$ and $\F$; see e.g. \cite[Theorem 4.25]{ferroni-matherne-vecchi}.
\end{proof}

Using the numerous formulas for dual Chow polynomials that we have presented here, we also give explicit formulas for all uniform matroids $\U_{r,n}$ of rank $r$ on ground set $[n]$.
    
\begin{proposition}
    The dual Chow polynomial of the uniform matroid $\U_{r,n}$ is
    \begin{align*}
        \H^*_{\U_{r,n}}(x)
        &=
        \binom{n-1}{r-1}
        +
        \sum_{j=0}^{r-2}
        \binom{n}{j}
        \binom{n-j-1}{r-j-1}
        A_j(x)\,
        (x+ \dots + x^{r-j-1}) \,.
    \end{align*}
    The dual augmented Chow polynomial of the uniform matroid $\U_{r,n}$ is
    \begin{align*}
        \F^*_{\U_{r,n}}(x)
        &=
        \binom{n-1}{r-1}
        +
        \sum_{j=0}^{r-1}
        \binom{n}{j}
        \binom{n-j-1}{r-j-1}
        A_j(x)\,
        (x+ \dots + x^{r-j}) \,.
    \end{align*}
\end{proposition}
\begin{proof}
    We group the chain formula of Theorem~\ref{thm:iterative} according to the
    penultimate flat of the chain. If $F<\one$ has rank $j$, then a chain from
    $\zero$ to $F$ contributes $(-1)^j\H^*_{\zero F}(x)$ to the unsigned part
    of the formula in Theorem~\ref{thm:iterative}. Hence
    \[
        \H^*_{\U_{r,n}}(x)
        =
        (-1)^r\mu_{\zero\one}
        +
        \sum_{F<\one}
        (-1)^{r+\rk(F)}
        \H^*_{\zero F}(x)\,
        \mu_{F\one}\,
        \frac{x^{\rho_{F\one}}-x}{x-1}.
    \]
    Recall that $\cL(\U_{r,n})$ consists of the set $[n]$ and all subsets of
    $[n]$ with cardinality at most $r-1$, ordered by inclusion. If $F$ has rank
    $j<r$, then $[\zero,F]$ is Boolean of rank $j$, and therefore
    $\H^*_{\zero F}(x)=A_j(x)$. Moreover,
    \[
        \mu_{F\one}=(-1)^{r-j}\binom{n-j-1}{r-j-1}.
    \]
    There are $\binom{n}{j}$ flats of rank $j$, and the top flat contributes
    $(-1)^r\mu_{\zero\one}=\binom{n-1}{r-1}$. Since the signs cancel for each
    $F<\one$, we obtain
    \[
        \H^*_{\U_{r,n}}(x)
        =
        \binom{n-1}{r-1}
        +
        \sum_{j=0}^{r-1}
        \binom{n}{j}\binom{n-j-1}{r-j-1}
        A_j(x)\frac{x^{r-j}-x}{x-1}.
    \]
    The summand with $j=r-1$ vanishes, and this proves the formula for
    $\H^*_{\U_{r,n}}(x)$.

    For $\F^*_{\U_{r,n}}(x)$ we use Proposition~\ref{prop:Hstar-from-Fstar}:
    \[
        \F^*_{\U_{r,n}}(x)
        =
        \sum_{F\in\cL(\U_{r,n})}
        \H^*_{\zero F}(x)(-x)^{\rho_{F\one}}\mu_{F\one}.
    \]
    The term $F=\one$ is $\H^*_{\U_{r,n}}(x)$. For flats of rank $j<r$, the
    remaining contribution is
    \[
        \binom{n}{j}
        \binom{n-j-1}{r-j-1}
        A_j(x)x^{r-j}.
    \]
    Adding this to the formula already obtained for $\H^*_{\U_{r,n}}(x)$
    replaces, for each $j<r$, the factor $x+\cdots+x^{r-j-1}$ by
    $x+\cdots+x^{r-j}$. This gives the claimed formula for
    $\F^*_{\U_{r,n}}(x)$.
\end{proof}

By \Cref{thm:gamma exp},
the $\gamma$-expansion of the dual Chow polynomial and the dual augmented Chow polynomial is determined by the flag $h$-vector of $\cL(\U_{r,n})$.
Following the definitions and standard interpretations of the flag $h$-vector in \cite{stanley-ec1}, we determine the flag $h$-vector of $\cL(\U_{r,n})$.

For a permutation $\sigma\in \mathfrak{S}_n$ let 
\[
    \Des(\sigma) := \{i\in [n-1]\mid \sigma(i) > \sigma(i+1) \}
\]
denote the set of descents. Similarly, let $\Asc(\sigma)$ denote the set of ascents.
For $D\subseteq [n-1]$, the Eulerian number $\mathrm{E}(n,D)$ is the number of permutations in $\mathfrak{S}_n$ having descent set $D$, that is,
\(
    \mathrm{E}(n,D) = \# \{ \sigma \in \mathfrak{S}_n \mid \Des(\sigma)=D \} \,.
\)

\begin{proposition}\label{prop:gamma-uniform}
    The $\gamma$-polynomial of the dual Chow polynomial of $\U_{r,n}$ is
    \[
        \gamma(\H^*_{\U_{r,n}};x)
        =
        \sum_{k=0}^{r-1}
        \binom{n-1-k}{r-1-k}
        p_{r-1,k}^{[r-1]\setminus\{1\}}(x).
    \]
    The $\gamma$-polynomial of the dual augmented Chow polynomial is
    \[
        \gamma(\F^*_{\U_{r,n}};x)
        =
        \sum_{k=0}^{r-1}
        \binom{n-1-k}{r-1-k}
        p_{r-1,k}^{[r-1]}(x).
    \]
\end{proposition}
\begin{proof}
    Following the notation and definitions in \cite{stanley-ec1}, the flag $f$-vector of the uniform matroid for $S\subseteq [r-1]$ is
    \(
        \alpha_{\cL(\U_{r,n})}(S) = \alpha_{\U_{n,n}}(S) \,,
    \)
    and thus 
    \[
        \beta_{\cL(\U_{r,n})}(S) = \beta_{\U_{n,n}}(S) = \mathrm{E}(n,S) \,.
    \]
    This proves the first identities of $\H^*_{\U_{r,n}}(x)$ and $\F^*_{\U_{r,n}}(x)$ respectively.
    For the second identities, observe that for $S\subseteq [r-1]$
    \[
        \mathrm{E}(n,S) = \sum_{\substack{\sigma\in \mathfrak{S}_r \\ \Des(\sigma) = S} }
        \binom{n-\sigma(r)}{r-\sigma(r)}\,.
    \]
    This follows, for instance, from an $R$-labeling argument that is discussed in detail in \cite[Section 3.2]{hoster}.
\end{proof}

Finally, we prove that both the dual Chow polynomial and the dual augmented Chow polynomial of a uniform matroid are real-rooted.
For the remainder of this section, we follow the notation and definition of \cite{hoster-stump}.
Let
\[
    p_{r-1,k}^{T}(x) = \sum_{\substack{w\in \mathfrak{S}_{r} \\ w(1) = k+1 \\ \Des(w) \subseteq T \\ \Des(w) \text{ stable} } } x^{\des(w)}\,,
\]
for $T\subseteq[r-1]$ (this notation appears with a different but related meaning in \cite[Section~3]{athanasiadis-douvropoulos-kalampogia}).
For a palindromic polynomial $f\in \mathbb{R}[x]$ of degree $d$ that has $\gamma$-expansion $f(x)=\sum_i \gamma_i \, x^i (1+x)^{d-2i}$, we define the associated $\gamma$-polynomial by $\gamma(f;x) = \sum_i \gamma_i x^i $\,.

\begin{corollary}
\label{cor:dual-chow-roots}
    The dual Chow polynomial $\H^*_{\U_{r,n}}(x)$ and the dual augmented Chow polynomial $\F^*_{\U_{r,n}}(x)$ of the uniform matroid $\U_{r,n}$ are real rooted.
\end{corollary}
\begin{proof}
    Let $T=[r-1]$ or $T=[r-1]\setminus\{1\}$.
    By \cite[Theorem 3.3]{hoster-stump} together with well-known properties of real-rooted polynomials, see e.g. \cite[Proposition 3.5]{wagner},
    any nonnegative linear combinations of the polynomials $p_{r-1,0}^T,p_{r-1,1}^T,\dots, p_{r-1,r-1}^T$ is again real-rooted.
    Thus,
    the $\gamma$-polynomials of the dual Chow polynomial and of the dual augmented Chow polynomial are both real-rooted. 
    By \cite{gal,branden2}, the same statement then holds for $\H^*_{\U_{r,n}}(x)$ and $\F^*_{\U_{r,n}}(x)$.
\end{proof}

\begin{remark}[Dual Chow polynomials of simplicial posets with nonnegative $h$-vector]
    Let $P$ be a simplicial poset of rank $r$ with nonnegative $h$-vector
    $(h_0,\dots,h_{r-1})$, and let $\aug^*(P)$ be the top augmentation of $P$.
    The same argument used in the proof of Corollary~\ref{cor:dual-chow-roots}
    shows that the dual Chow polynomial and the dual augmented Chow polynomial
    of $\aug^*(P)$ are real-rooted, since their associated $\gamma$-polynomials are
    \[
        \gamma(\H^*_{\aug^*(P)};x)
        =
        \sum_{k=0}^{r-1}h_{r-1-k}\,
        p_{r-1,k}^{[r-1]\setminus\{1\}}(x)
    \]
    and
    \[
        \gamma(\F^*_{\aug^*(P)};x)
        =
        \sum_{k=0}^{r-1}h_{r-1-k}\,
        p_{r-1,k}^{[r-1]}(x).
    \]
    These are the same formulas as in \cite[Lemma~3.1]{hoster-stump} for
    $\chi$-Chow polynomials, but with $h$-vector entries reversed.
\end{remark}

\section{Final remarks}\label{sec:final-remarks}

\subsection{The geometric picture}

This section assumes that the reader is familiar with the standard terminology on the Hodge theory of matroids and polymatroids; see, for instance, \cite{adiprasito-huh-katz,semismall,pagaria-pezzoli}. When $P=\mathcal{L}(\M)$ is the lattice of flats of a loopless matroid $\M$, the Chow polynomial $\H_P(x)=\H_{\M}(x)$ can be interpreted as the Hilbert--Poincar\'e series of the Chow ring $\uCH(\M)$ of $\M$; see \cite{feichtner-yuzvinsky,semismall}. The ring $\uCH(\M)$ is a standard graded quotient of a polynomial ring with variables indexed by the non-empty proper flats of $\M$. It is natural to ask whether there is a comparable algebraic object whose graded dimensions are given by the dual Chow polynomial. One should not expect this object to be a standard graded algebra with one-dimensional degree-zero part: indeed, for a rank $r$ matroid,
\[
    [x^0]\H^*_{\M}(x)=(-1)^r\mu_{\M},
\]
which is the M\"obius invariant of $\M$, and is rarely equal to $1$.

The chain formula of Theorem~\ref{thm:iterative} suggests that the relevant object should be constructed from flags of flats and from top Orlik--Solomon spaces of the interval minors appearing along those flags. This is very close in spirit to the Leray model of Bibby, Denham, and Feichtner~\cite{bibbi-denham-feichtner}. More precisely, one should use the projective version of their Leray model, taken with respect to the maximal building set. This is a bigraded differential graded algebra whose bidegree $(\bullet,0)$ part is the Chow ring and whose other bidegrees contain Orlik--Solomon type information coming from local minors. For an interval $[F,G]$ in $\cL(\M)$, let
\(
    \operatorname{OS}(F,G):=\operatorname{OS}\bigl((\M|G)/F\bigr)
\)
be the Orlik--Solomon algebra of the interval minor, and let
\(
    \operatorname{OS}^{\mathrm{top}}(F,G)
\)
denote its top-degree part. Equivalently, in the projective Leray model one may use the top piece of the corresponding reduced Orlik--Solomon factor; these top pieces are canonically identified. In either language,
\[
    \dim \operatorname{OS}^{\mathrm{top}}(F,G)
    =
    (-1)^{\rk(G)-\rk(F)}\mu_{FG}.
\]

Thus one expects the dual Chow polynomial to be the Hilbert series of a ``local-top'' part of the projective Leray model: namely, the part obtained by retaining, along each flag of flats, only the top Orlik--Solomon piece of every local interval minor. At the level of graded vector spaces, the expected associated graded object would have the form
\[
    \bigoplus_{F_0\subsetneq \cdots\subsetneq F_m=E}
    \operatorname{OS}^{\mathrm{top}}(\varnothing,F_0)
    \otimes
    \bigotimes_{j=1}^{m}
    \left(
        \operatorname{OS}^{\mathrm{top}}(F_{j-1},F_j)
        \otimes
        x\Q[x]/(x^{\rk(F_j)-\rk(F_{j-1})})
    \right),
\]
where all Orlik--Solomon factors are treated as coefficient spaces in Chow degree $0$, and where $x\Q[x]/(x^d)$ has Hilbert series $x+x^2+\cdots+x^{d-1}$. Thanks to Theorem~\ref{thm:iterative}, taking Hilbert series in this expression gives exactly $\H^*_{\M}(x)$.
Indeed, the top Orlik--Solomon factors contribute the absolute values of the M\"obius numbers in the chain formula, while the truncated polynomial factors contribute the terms $x+\cdots+x^{d-1}$.

This discussion should be regarded as evidence for the existence of a natural \emph{dual Chow module}, rather than as a complete construction of such a module. The subtle point is the module structure. Although the projective Leray model is naturally a module over its Chow part, it is not immediate that the local-top piece described above is literally a submodule. The natural expectation is that, after passing to an appropriate flag associated graded or quotient, this local-top construction becomes a well-defined module over $\uCH(\M)$.

\begin{question}[Dual Chow module]\label{ques:dual-chow-module}
    Can the local-top part of the projective Leray model be made into a natural graded $\uCH(\M)$-module $\uCH^{\star}(\M)$ whose Hilbert series is $\H^*_{\M}(x)$?
\end{question}

In the realizable case, such a construction should be related to taking iterated top logarithmic residues along the boundary strata of the wonderful compactification. In the non-realizable case, the projective Leray model should provide the purely combinatorial replacement for this geometric picture.

Assuming that such a module $\uCH^{\star}(\M)$ can be constructed, one is led to the following Hodge-theoretic question. Since the expected dual Chow module would be a module rather than a ring, the correct analogue of the usual K\"ahler package should not be formulated in terms of a trace map on an algebra. Instead, one should ask whether it is a polarized Lefschetz module over the Chow ring.

\begin{question}[K\"ahler package for the dual Chow module]\label{op:kahler-dual-chow-module}
    Does there exist a natural dual Chow module for $\M$ admitting a Lefschetz package with respect to the action of $\uCH(\M)$?
\end{question}

More precisely, for every ample class $\ell\in\uCH^1(\M)$, should multiplication by $\ell$ be a Lefschetz operator on this module? That is, should there be isomorphisms
\[
    \ell^{\,r-1-2j}:
    \uCH^{\star,j}(\M)
    \xrightarrow{\ \sim\ }
    \uCH^{\star,r-1-j}(\M)
    \qquad
    \text{for all } j<\tfrac{\rk(\M)-1}{2}?
\]
Moreover, one may ask whether there is a natural nondegenerate pairing
\[
    \langle-, -\rangle_{\M}:
    \uCH^{\star,j}(\M)\times
    \uCH^{\star,r-1-j}(\M)
    \longrightarrow \Q
\]
with respect to which multiplication by $\ell$ is self-adjoint, and for which the corresponding Hodge--Riemann bilinear relations hold on the primitive subspaces.

There are several reasons to expect such a picture. First, the polynomial $\H^*_{\M}(x)$ is palindromic of degree $r-1$, so the numerical shadow of Poincar\'e duality is already present. Second, the deletion formula of Theorem~\ref{thm:dual-chow-deletion} expresses the dual Chow polynomial as a positive sum of products of smaller such polynomials, in a form reminiscent of semi-small decompositions. Third, the degree-zero part of the expected module should be the top Orlik--Solomon space of $\M$, whose dimension is the M\"obius invariant $(-1)^r\mu_{\M}$; by palindromicity, the same dimension appears in top degree. Thus the top Orlik--Solomon part appears to play, for the dual Chow polynomial, the role that the one-dimensional orientation space plays for the usual Chow ring.

One may also ask for an augmented counterpart of this story. The corresponding ``dual right augmented Chow module'' should be related to the augmented Chow ring $\CH(\M)$, and hence to the graded M\"obius algebra. This version may be the more natural setting for functoriality and for comparison with the singular Hodge theory of matroids of Braden, Huh, Matherne, Proudfoot, and Wang \cite{braden-huh-matherne-proudfoot-wang}.

\subsection{Real-rootedness for dual Chow polynomials of matroids}

A well-known conjecture posed by Huh--Stevens \cite[Conjecture~4.3.1]{stevens-bachelor} and, independently, Ferroni--Schr\"oter \cite[Conjecture~8.18]{ferroni-schroter} postulates that for every matroid $\M$ the Chow polynomial of $\mathcal{L}(\M)$, that is, the Hilbert--Poincar\'e series of the Chow ring of $\M$, has only nonpositive real zeros. There has been steady progress towards this conjecture in recent work, see \cite{branden-vecchi-uniform,branden-vecchi2,hoster-stump,coron-ferroni-li}.

We formulate the following analogue of that conjecture.

\begin{conjecture}
    For every matroid $\M$, the dual Chow polynomial of $\mathcal{L}(\M)$ has only nonpositive real zeros.
\end{conjecture}

We further conjecture that the same happens for the dual right augmented Chow polynomial of $\mathcal{L}(\M)$, and that its roots are interlaced by those of the dual Chow polynomial.

\begin{remark}
    In an early version of \cite{ferroni-matherne-vecchi} it was conjectured that real-rootedness of Chow polynomials might hold not just for geometric lattices, but also for the class of Cohen--Macaulay posets. A counterexample to that conjecture was found by the authors of \cite{ferroni-matherne-vecchi} in mid 2026 (see \cite[Figure~4]{ferroni-matherne-vecchi}). It is reasonable to ask for a counterexample to the real-rootedness of dual Chow polynomials of Cohen--Macaulay posets. In Figure~\ref{fig:cm-dual-chow-not-real-rooted} we depict one such poset, which is in fact an EL-shellable poset as can be checked by a brute force computer check. The dual Chow polynomial of this poset is
    \[4 x^{5} + 39 x^{4} + 120 x^{3} + 120 x^{2} + 39 x + 4
    \]
    which has four complex zeros near $-4.13\pm 0.35\, i$ and $-0.24\pm 0.02\, i$.
    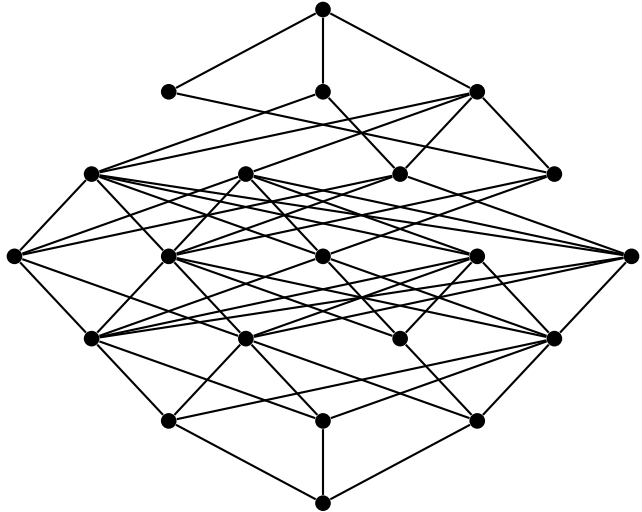
\begin{figure}[ht]
\centering
\begin{tikzpicture}
[scale=0.68,auto=center,
 every node/.style={circle,scale=0.8,fill=black,inner sep=2.6pt}]
\tikzstyle{edges}=[thick];

\node[] (v0) at (0,0) {};

\node[] (v1) at (-3,1.6) {};
\node[] (v2) at (0,1.6) {};
\node[] (v3) at (3,1.6) {};

\node[] (v4) at (-4.5,3.2) {};
\node[] (v5) at (-1.5,3.2) {};
\node[] (v6) at (1.5,3.2) {};
\node[] (v7) at (4.5,3.2) {};

\node[] (v8)  at (-6,4.8) {};
\node[] (v9)  at (-3,4.8) {};
\node[] (v10) at (0,4.8) {};
\node[] (v11) at (3,4.8) {};
\node[] (v12) at (6,4.8) {};

\node[] (v13) at (-4.5,6.4) {};
\node[] (v14) at (-1.5,6.4) {};
\node[] (v15) at (1.5,6.4) {};
\node[] (v16) at (4.5,6.4) {};

\node[] (v17) at (-3,8) {};
\node[] (v18) at (0,8) {};
\node[] (v19) at (3,8) {};

\node[] (v20) at (0,9.6) {};

\draw[edges] (v0) -- (v1);
\draw[edges] (v0) -- (v2);
\draw[edges] (v0) -- (v3);

\draw[edges] (v1) -- (v4);
\draw[edges] (v1) -- (v5);
\draw[edges] (v1) -- (v7);

\draw[edges] (v2) -- (v4);
\draw[edges] (v2) -- (v5);
\draw[edges] (v2) -- (v7);

\draw[edges] (v3) -- (v5);
\draw[edges] (v3) -- (v6);
\draw[edges] (v3) -- (v7);

\draw[edges] (v4) -- (v8);
\draw[edges] (v4) -- (v9);
\draw[edges] (v4) -- (v10);
\draw[edges] (v4) -- (v11);
\draw[edges] (v4) -- (v12);

\draw[edges] (v5) -- (v8);
\draw[edges] (v5) -- (v9);
\draw[edges] (v5) -- (v11);
\draw[edges] (v5) -- (v12);

\draw[edges] (v6) -- (v9);
\draw[edges] (v6) -- (v10);
\draw[edges] (v6) -- (v11);

\draw[edges] (v7) -- (v9);
\draw[edges] (v7) -- (v10);
\draw[edges] (v7) -- (v11);
\draw[edges] (v7) -- (v12);

\draw[edges] (v8) -- (v13);
\draw[edges] (v8) -- (v14);
\draw[edges] (v8) -- (v15);

\draw[edges] (v9) -- (v13);
\draw[edges] (v9) -- (v14);
\draw[edges] (v9) -- (v15);
\draw[edges] (v9) -- (v16);

\draw[edges] (v10) -- (v13);
\draw[edges] (v10) -- (v14);
\draw[edges] (v10) -- (v16);

\draw[edges] (v11) -- (v13);
\draw[edges] (v11) -- (v14);

\draw[edges] (v12) -- (v13);
\draw[edges] (v12) -- (v14);
\draw[edges] (v12) -- (v15);

\draw[edges] (v13) -- (v18);
\draw[edges] (v13) -- (v19);

\draw[edges] (v14) -- (v19);

\draw[edges] (v15) -- (v18);
\draw[edges] (v15) -- (v19);

\draw[edges] (v16) -- (v17);
\draw[edges] (v16) -- (v19);

\draw[edges] (v17) -- (v20);
\draw[edges] (v18) -- (v20);
\draw[edges] (v19) -- (v20);

\end{tikzpicture}
\caption{An EL-shellable (thus Cohen--Macaulay) poset whose dual Chow polynomial is not real-rooted.}
\label{fig:cm-dual-chow-not-real-rooted}
\end{figure}
\end{remark}

\subsection{Relationship with the inverse KL polynomial of a matroid}

In work of Gao and Xie \cite{gao-xie} and Gao, Ruan, and Xie \cite{gao-ruan-xie} the \emph{inverse Kazhdan--Lusztig polynomial} $Q_{\M}(x)$ and the \emph{inverse $Z$-polynomial} $Y_{\M}(x)$ of a matroid $\M$ are thoroughly studied. Using the notation of our Section~\ref{sec:preliminaries}, and thanks to Proposition~\ref{prop:kls-and-zeta-of-rev-kernel}, one has:
    \[ Q_{\M}(x) := g^*_{\mathcal{L}(\M)}(x)\qquad \text{ and } \qquad Y_{\M}(x) = Z^*_{\mathcal{L}(\M)}(x).\]

The dual Chow polynomial $\H^*_{\M}(x)$, and its right augmented counterpart $\F^*_{\M}(x)$ can be expressed in terms of the polynomials displayed above, thanks to the numerical canonical decompositions proved by Ferroni, Matherne, and Vecchi in \cite{ferroni-matherne-vecchi}.

Braden, Ferroni, Matherne, and Nepal \cite{braden-ferroni-matherne-nepal} found deletion formulas for $Q_{\M}(x)$ and $Y_{\M}(x)$ which substantially differ from the deletion formulas for $\H^*_{\M}(x)$ and $\F^*_{\M}(x)$ proved in Theorem~\ref{thm:dual-chow-deletion} (specifically, where we see the sets $\underline{\mathscr{S}}_i$, they have a set called $\mathscr{T}_i$).

At the moment we lack a satisfactory explanation of this difference. Furthermore, the methodology used in the proofs in \cite{ferroni-matherne-vecchi} could potentially yield a proof of an alternative deletion formula for $\H^*_{\M}(x)$ and $\F^*_{\M}(x)$ or, at least, a substantially different proof of the ones described in Section~\ref{sec:matroids}.

\subsection{Beyond deletion formulas in matroid theory}

Aside from the Chow polynomial, the left augmented Chow polynomial, and the two inverse Kazhdan--Lusztig type polynomials $Q_{\M}(x)$ and $Y_{\M}(x)$ discussed above, many other polynomial invariants in matroid theory admit deletion formulas. Examples include the Tutte polynomial, the Kazhdan--Lusztig polynomial and the $Z$-polynomial \cite{braden-vysogorets}, as well as the polynomial studied in \cite[Theorem~A.1]{berget-spink-tseng}.

The deletion formulas for the extended $\aaa\bbb$-index obtained in Theorem~\ref{thm:extended-ab-deletion} also specialize to natural identities outside the scope of Chow polynomials. We record one such specialization, concerning the $h$-polynomial of the Bergman complex.

Recall that if $\M$ is a loopless matroid, its Bergman complex $\Delta(\widehat{\cL}(\M))$
is the order complex of the proper part $\widehat{\cL}(\M):=\cL(\M)\setminus\{\widehat{0},\widehat{1}\}$.
We denote its $h$-polynomial by $h_{\M}(x)$. Equivalently, if $\Psi_{\cL(\M)}(\aaa,\bbb)$ denotes the ordinary $\aaa\bbb$-index of the lattice of flats, then
\[
    h_{\M}(x)=\Psi_{\cL(\M)}(1,x).
\]

\begin{theorem}\label{thm:bergman-h-polynomial-deletion}
    Let $\M$ be a loopless matroid on $E$, and let $i\in E$ be an element that is not a coloop. Then
    \[
        h_{\M}(x)
        =
        h_{\M\setminus i}(x)
        +
        x\sum_{F\in\mathscr{S}_i}
        h_{\M|F}(x)\,
        h_{\M/(F\cup\{i\})}(x).
    \]
\end{theorem}

\begin{proof}
    We specialize the identity in Theorem~\ref{thm:ab-deletion} at $\aaa=1$ and $\bbb=x$.
    Since
    \(
        \Psi_{\cL(\N)}(1,x)
        =
        h_{\N}(x)
    \)
    for every minor $\N$ of $\M$,
    we have
    \[
        h_{\M}(x)
        = h_{\M\setminus i} 
            + x h_{\M/i} 
            + \sum_{F\in \underline{\mathscr{S}}_i } 
                h_{\M|F}\, x\, h_{\M/(F\cup\{i\})} \,.
    \]
    After minor term orderings, we obtain the claimed identity.
\end{proof}

This is strikingly close to the semi-small decomposition for the Chow polynomial $\H_{\M}(x)$ (see \cite[Corollary~3.24]{ferroni-matherne-vecchi}), with the main difference that the Bergman $h$-polynomial has the additional contraction term indexed by $F=\varnothing$. Moreover, both the polynomials $h_{\M}(x)$ and $\H_{\M}(x)$ take the same values at the base cases: for Boolean matroids they both correspond to the Eulerian polynomial of degree $\rk(\M)-1$. It is somewhat surprising that such a close parallel seems to have gone unnoticed, especially in light of a conjecture of Athanasiadis and Kalampogia-Evangelinou \cite{athanasiadis-kalampogia}, which predicts that $h_{\M}(x)$ has only nonpositive real zeros for every matroid $\M$.

\bibliographystyle{amsalpha}
\bibliography{bibliography}

\end{document}